\documentclass{article}
\usepackage{listings}
\usepackage{amsmath, amsthm, amssymb}
\newcommand{\R}{\mathbb R}
\newcommand{\E}{\mathbb E}
\newtheorem{prop}{Proposition}

\newtheorem{lemma}{Lemma}
\newtheorem{remark}{Remark}
\newtheorem{teo}{Theorem}

\usepackage{graphics}
\usepackage{graphicx}
\usepackage[scale=0.75]{geometry}
\usepackage[utf8]{inputenc}
\usepackage[english]{babel}

\usepackage{listings}
\graphicspath{ {./images/} }
\usepackage[output-decimal-marker={,}]{siunitx}
\usepackage{booktabs}
\usepackage{enumitem}
\allowdisplaybreaks
\usepackage{relsize}

\usepackage{xcolor}
\usepackage{tcolorbox}
\usepackage{verbatim}
\usepackage{bm}

\newcounter{hypoconbisfl}
\newcounter{saveconbisfl}
\newcommand\debutL{\begin{list} {\textbf{A\arabic{hypoconbisfl}}}{\usecounter{hypoconbisfl}}\setcounter{hypoconbisfl}{\value{saveconbisfl}}}
\newcommand\finL{\end{list}\setcounter{saveconbisfl}{\value{hypoconbisfl}}}

\newcounter{hypoconbisf}
\newcounter{saveconbisf}
\newcommand\debutP{\begin{list} {\textbf{B\arabic{hypoconbisf}}}{\usecounter{hypoconbisf}}\setcounter{hypoconbisf}{\value{saveconbisf}}}
\newcommand\finP{\end{list}\setcounter{saveconbisf}{\value{hypoconbisf}}}

\newcounter{hypoconbisx}
\newcounter{saveconbisx}
\newcommand\debutTX{\begin{list} {\textbf{TX\arabic{hypoconbisx}}}{\usecounter{hypoconbisx}}\setcounter{hypoconbisx}{\value{saveconbisx}}}
\newcommand\finTX{\end{list}\setcounter{saveconbisx}{\value{hypoconbisx}}}

\makeatletter
\newcommand{\mathleft}{\@fleqntrue\@mathmargin0pt}
\newcommand{\mathcenter}{\@fleqnfalse}

\makeatother
\title{Convergence to the uniform distribution of vectors of partial sums modulo one with a common factor}
\usepackage{mathtools}

\DeclarePairedDelimiter\floor{\lfloor}{\rfloor}
\author{
	Roberta Flenghi$^1$\\
	\texttt{roberta.flenghi@enpc.fr}
	\and
	Benjamin Jourdain$^1$\\
	\texttt{benjamin.jourdain@enpc.fr}
}

\date{%
	$^1$Cermics, \'Ecole des Ponts, INRIA, Marne-la-Vall\'ee, France.\\}

\newcommand\blfootnote[1]{%
	\begingroup
	\renewcommand\thefootnote{}\footnote{#1}%
	\addtocounter{footnote}{-1}%
	\endgroup
}
\usepackage{mathtools}

\begin{document}
	\maketitle
	\blfootnote{``This work is
		supported by the french National Research Agency under the grant
		ANR-21-CE40-0006 (SINEQ).''}
\begin{abstract}
    In this work, we prove the joint convergence in distribution of $q$ variables modulo one obtained as partial sums of a sequence of i.i.d. square integrable random variables multiplied by a common factor given by some function of an empirical mean of the same sequence. The limit is uniformy distributed over $[0,1]^q$.  To deal with the coupling introduced by the common factor, we assume that the joint distribution of the random variables has a non zero component absolutely continuous with respect to the Lebesgue measure, so that the convergence in the central limit theorem for this sequence holds in total variation distance. While our result provides a generalization of Benford's law to a data adapted mantissa, our main motivation is the derivation of a central limit theorem for the stratified resampling mechanism, which is performed in the companion paper \cite{echant}.  
\end{abstract}	\section{Introduction}\label{introduction}
	Given $\left(Y_i\right)_{i\geq 1}$  a sequence of i.i.d. square integrable random variables and a measurable real valued function $\phi$, we are going to give sufficient conditions for the convergence in distribution to the uniform law on $[0,1]$ of 
	\begin{align}\label{purpose0}
	\left\{\phi\left(\frac{1}{M}\sum_{i=1}^{\beta_M}Y_i\right) \left( Y_1+\cdots+Y_{M}\right)\right\},
	\end{align}for $\left(\beta_M \right)_{M\geq 1} \subseteq \mathbb{N}^*$ such that $\lim\limits_{M\rightarrow \infty}\sqrt{M}\left( \frac{\beta_M}{M}- \beta\right) =0$  with $\beta>1$. Here $\{x\}$ denotes the fractional part of the real number $x$ given by $\{x\}=x-\lfloor x\rfloor$ where $\lfloor x\rfloor$ is the integer such that $\lfloor x\rfloor\le x<\lfloor x\rfloor+1$ and we also define $\lceil x\rceil$ as the integer such that $\lceil x\rceil-1<x\le\lceil x\rceil$.  
	Our main motivation for considering \eqref{purpose0} comes from the derivation of a central limit theorem for the stratified resampling mechanism. Before giving more details about this particular application, let us review the existing literature which addresses the case of a constant function $\phi$ with the derivation of Benford's law as a common motivation. 
	
The convergence in distribution of the sequence  $(V_M=\left\lbrace Y_1+\cdots+Y_M\right\rbrace )_{M\geq 1}$ of sums of random variables defined modulo $1$ to the uniform distribution on $[0,1]$ has been studied by many researchers using Fourier analysis.
	In 1939, L\'{e}vy \cite{Levy}  gave necessary and sufficient conditions for this convergence when the $Y_i$ are i.i.d..
	In 1986, St\"{o}rmer \cite{Stormer} provided sufficient conditions in terms of the distribution functions of the $Y_i$ for the convergence to hold when the $Y_i$ are merely independent. In 2007, under the assumption of independent absolutely continuous $Y_i$,   Miller and Nigrini \cite{Miller.Nigrini} 
        characterized
        the convergence of $(V_M)_{M\geq 1}$ in $L^1\left( \left[0,1 \right] \right)$. In 2010, Szewczak \cite{Szewczak} generalized the above results by getting rid of the hypothesis of independence of the $Y_i$. In particular he proved that, if the absolute value of the characteristic function of $Y_1+\cdots +Y_M$ satisfies a certain growth condition and if the set $\left\lbrace n \in \mathbb{N}\setminus \left\lbrace0 \right\rbrace : \left|\mathbb{E}\left( e^{2\pi i n Y_1}\right)  \right|=1  \right\rbrace $ is empty, then $(V_M)_{M\geq 1}$ converges in distribution to a uniform random variable on $[0,1]$.

        Let us now briefly introduce Benford's law and the problem of the distribution of the leading digits of products of random variables (see for instance \cite{Dia}). Benford's law in base $b>1$ is the probability measure $\mu_{b}$ on the interval $\left[ 1,b\right) $ defined by 
	\begin{align*}
	\mu_{b}\left( \left[ 1,a\right) \right) = \log_{b}a,\quad \quad \forall a\in[1,b). 
	\end{align*}
	The mantissa in base $b$ of a positive real number $x$ is the unique number $\mathcal{M}_b\left(x \right)$ in $\left[ 1,b\right)$ such that $x = \mathcal{M}_b\left(x \right)\times b^{\lfloor \log_b(x)\rfloor}$. Given a sequence of positive random variables $\left(X_i \right)_{i\geq 1}$, many researchers (\cite{shatte}, \cite{boyle}, \cite{shatte2}, \cite{ley}) were interested in studying the weak  convergence as $M\rightarrow \infty$ of the law of $\mathcal{M}_b\left(\prod\limits_{i=1}^{M}X_i \right)$  to  $\mu_{b}$. Since $\log_{b}\mathcal{M}_b\left(x\right)=\left\lbrace \log_{b}\left(x \right)  \right\rbrace$ for each positive real number $x$, one has $$\log_{b}\mathcal{M}_b\left(\prod\limits_{i=1}^{M}X_i \right)=\left\lbrace \sum_{i=1}^M\log_{b}\left(X_i \right)  \right\rbrace,$$
and this convergence is equivalent to the weak convergence of the partial sums modulo 1 of the random variables $(\log_{b}\left(X_i \right))_{i\ge 1}$ by continuity of $[1,b)\ni z\mapsto \log_b(z)$ and its inverse.\\

The introduction of a non-constant function $\phi$ in our work permits to address the choice of a data dependent mantissa for instance given by the geometric mean $\hat b_M=\exp\left(\frac{1}{\beta M}\sum_{i=1}^{\beta_M}\ln(X_i)\right)$. Indeed,
$$\log_{\hat b_M}\mathcal{M}_{\hat b_M}\left(\prod\limits_{i=1}^{M}X_i \right)=\left\lbrace \frac{\sum_{i=1}^M\ln(X_i)}{\frac 1{\beta M}\sum_{i=1}^{\beta_M}\ln(X_i)}\right\rbrace=\left\lbrace \phi\left(\frac{1}{M}\sum_{i=1}^{\beta_M}\ln(X_i)\right)\sum_{i=1}^M\ln(X_i)\right\rbrace\mbox{ where }\phi(x)=\frac 1{\beta x}.$$
Compared to previous works, the main difficulty that we have to address comes from the coupling between the variables introduced through the common factor $\phi\left(\frac{1}{M}\sum_{i=1}^{\beta_M}Y_i\right)$.  To overcome this difficulty, we assume that the $Y_i$ are i.i.d. according to a common distribution with a non-zero component absolutely continuous with respect to the Lebesgue measure. This allows us to apply one of the key results for our proof, the convergence in total variation in the central limit theorem, that we now recall. Let $F$ be a centered square-integrable random vector in $\mathbb{R}^n$ with identity covariance matrix and let $(F_i)_{i\ge 1}$ independent copies of $F$. Under the assumption that the law of $F$ has an absolutely continuous component, Prohorov \cite{prohorov} in the one-dimensional case $n=1$ and Bally and Caramellino \cite{ballycaramellino} in the multidimensional case, proved that the total variation distance between the distribution of $\frac{1}{\sqrt{M}}\sum_{i=1}^{M}F_i$ and the standard Gaussian law in $\R^n$ goes to $0$ as $M\to\infty$.
        
	The study of the joint convergence of several partial sums modulo $1$ with a common factor and an additional component satisfying a central limit theorem does not add further significant difficulties and is useful in the derivation of the central limit theorem for the stratified resampling mechanism. That is why we address the convergence in distribution of 
	
	\begin{align}\label{purpose}
	\left(\left( \left\lbrace  \phi\left(\frac{1}{M}\sum_{m=1}^{\beta^{q+1}_M}Y_m\right) \left( Y_1+\cdots+Y_{\beta_M^i}\right)\right\rbrace\right)_{1\leq i \leq q} ,\sqrt{M}\left(\phi\left(\frac{1}{M}\sum_{m=1}^{\beta^{q+1}_M}Y_m\right)\tfrac{1}{M}\sum_{i=1}^{\beta^{q+1}_M}Z_i-\theta \right) \right)
	\end{align}
for $(\beta_M^1,\cdots,\beta_M^{q+1} )_{M\geq 1}\subseteq \mathbb{N}^{q+1}$ such that  $\lim\limits_{M\rightarrow \infty}\sqrt{M}\left( \frac{\beta_M^i}{M}- \beta^i\right) =0$  with $0<\beta^1< \cdots <\beta^{q+1}$ and $\left(Z_i \right)_{i\geq 1}$ such that the sequence $((Y_i,Z_i))_{i\ge 1}$ also is i.i.d. and the last component where $\theta$ is a constant (typically equal to $\beta^{q+1}\phi\left(\beta^{q+1}\E(Y_1)\right)\E(Z_1)$) converges in distribution to $T$ as $M\to\infty$. We give mild conditions ensuring that the full vectors converge in distribution to $(U,T)$ where $U$ is uniformly distributed on $[0,1]^q$ and independent of $T$.
	\subsubsection*{Main motivation}
        Let us now come back to our main motivation. We are interested in providing a central limit theorem for the stratified resampling scheme under the simplifying assumption that the $\R^d$-valued initial drawings $(X_m)_{m\ge 1}$ are independent and identically distributed and weighted proportionally to their image by some measurable function $g:\R^d\to (0,+\infty)$. For $M\ge 1$, resampling schemes (see for instance \cite{Douc}) permit to replace the probability measure $\frac{\sum_{m=1}^Mg(X_m)\delta_{X_m}}{\sum_{m=1}^Mg(X_m)}$ with non equal weights by some empirical measure $\frac 1 M\sum_{m=1}^M\delta_{\xi_m^M}$ with the same conditional expectation given ${\cal F}=\sigma((X_m)_{m\ge 1})$.  For $(U_m)_{m\ge 1}$ an independent sequence of independent random variables uniformly distributed on $(0,1)$, the stratified resampling scheme consists in setting 
\begin{align*}
	\xi_m^M = \sum_{i=1}^{M}1_{\left\lbrace \bar{S}^M_{i-1} <m-U_m\leq \bar{S}^M_{i}\right\rbrace }X_{i}\mbox{ for }m\in\{1,\cdots,M\}\mbox{ with }\bar{S}^M_j=\frac{M\sum_{m=1}^j g(X_m)}{\sum_{m=1}^M g(X_m)}\mbox{ for }j\in\{1,\cdots,M\},
\end{align*}
under the convention $\bar{S}^M_0=0$.
For $f:\R^d\to\R$ measurable, we have
\begin{equation}
   {\mathbb E}\left(f(\xi_m^M)|{\cal F}\right)=\sum_{i=1}^Mf(X_i)\int_{m-1}^m 1_{\{\bar S_{i-1}^M<u\le \bar S^M_i\}}du.\label{espcond}
 \end{equation}
 The central limit theorem deals with the asymptotic behaviour of $\frac{1}{\sqrt{M}}\sum_{m=1}^{M}\left( f(\xi_m^{M})- \frac{\mathbb{E}\left(f\left(X_1 \right) g\left(X_1\right)  \right)}{\E(g(X_1))} \right)$ as $M\to\infty$. Let us explain how the computation of the asymptotic variance is related to our main result.
By a standard decomposition of the variance and the conditional independence of $(\xi^M_m)_{1\le m\le M}$ given ${\cal F}$,
\begin{align*}
	&\textrm{Var}\left(\dfrac{1}{\sqrt{M}}\sum_{m=1}^{M}f(\xi_m^{M})\right)= \textrm{Var}\left(\mathbb{E}\left(\dfrac{1}{\sqrt{M}}\sum_{m=1}^{M}f(\xi_m^{M})\mathrel{\Big|}{\cal F}\right) \right)+\mathbb{E}\left(\textrm{Var}\left(
	\dfrac{1}{\sqrt{M}}\sum_{m=1}^{M}f(\xi_m^{M})\mathrel{\Big|}
	{\cal F}\right)\right)   \\
	&=\textrm{Var}\left( \sqrt{M}\dfrac{\sum_{m=1}^{M}g(X_{m})f(X_{m})}{\sum_{m=1}^{M}g(X_{m})} \right)+ \mathbb{E}\left(\dfrac{\sum_{m=1}^{M}g(X_{m})f^2(X_{m})}{\sum_{m=1}^{M}g(X_{m})}\right)-\mathbb{E}\left(\frac{1}{M}\sum_{m=1}^{M} \mathbb{E}\left(f(\xi_m^{M}) \mathrel{\Big|}{\cal F} \right)^2\right).	
	\end{align*}
 Since the asymptotic behaviour as $M\to\infty$ of the first two terms in the right-hand side can be analysed using standard arguments, we focus on that of the third term. 
Using \eqref{espcond} and $\bar S_M^M=M$, we get that $\sum_{m=1}^M {\mathbb E}\left(f(\xi_m^M)|{\cal F}\right)^2$ is equal to
\begin{align*}
&\sum_{k=0}^{M-1}\left(1+1_{\{k\ge 1\}}\right)\sum_{i=1}^{M-k}f(X_i)f(X_{i+k})\sum_{m\ge \lfloor \bar S^M_{i-1}\rfloor+1}\int_{m-1}^m1_{\{\bar S_{i-1}^M<u\le \bar S^M_i\}}du
\int_{m-1}^m1_{\{\bar S_{i+k-1}^M<u\le \bar S^M_{i+k}\}}du.\end{align*}
Since $\frac 1M\sum_{i=1}^M\gamma_i=\int_0^1\gamma_{\lceil \alpha M\rceil}d\alpha$, we deduce that $\frac 1 M\sum_{m=1}^M {\mathbb E}\left(f(\xi_m^M)|{\cal F}\right)^2$ is equal to the sum over $k\in{\mathbb N}$ of
\begin{align*}
  \left(1+1_{\{k\ge 1\}}\right)\int_0^11_{\{\lceil \alpha M\rceil\le M-k\}}f(X_{\lceil \alpha M\rceil})f(X_{\lceil \alpha M\rceil+k})\psi_k\left(\left\lbrace\bar S^M_{\lceil\alpha M \rceil-1}\right\rbrace,\tfrac{Mg\left(X_{\lceil\alpha M \rceil} \right)}{\sum_{m=1}^{M}g\left(X_m \right)},\cdots,\tfrac{Mg\left(X_{\lceil\alpha M \rceil+k} \right)}{\sum_{m=1}^{M}g\left(X_m \right)} \right)d\alpha
\end{align*}
where $\psi_k(u_0,w_1,\cdots,w_{k+1})=\sum\limits_{m\geq 1}\int_{m-1}^{m}1_{\left\lbrace u_0 <u\leq u_0 +w_{1}\right\rbrace}du \int_{m-1}^{m}1_{\left\lbrace u_0+\sum\limits_{\ell=1}^{k}w_{\ell} <u\leq u_0+\sum\limits_{\ell=1}^{k+1}w_{\ell} \right\rbrace}du$.
That is why, in order to compute the asymptotic variance of $\dfrac{1}{\sqrt{M}}\sum_{m=1}^{M}f(\xi_m^{M})$, it is very useful to understand the behaviour as $M\rightarrow \infty$ of $\left\lbrace \bar S^M_{\lceil\alpha M \rceil-1}\right\rbrace$ which is equal to the first component in (\ref{purpose}) when $q=1$, $Y_i=g\left(X_i \right), \phi(x)= \frac{1}{x}$, $\beta_M^1 = \lceil\alpha M \rceil-1$ and $\beta_M^2 =M$.\\
	 
	Regarding the Central Limit Theorem, the characteristic function of $\frac{1}{\sqrt{M}}\sum_{m=1}^{M}\left( f(\xi_m^{M})- \frac{\mathbb{E}\left(f\left(X_1 \right) g\left(X_1\right)  \right)}{\E(g(X_1))} \right)$  writes for $u\in\R$
	 	\begin{align*} 
	 &\mathbb{E}\left(e^{iu\sqrt{M}\left( \frac{\sum_{m=1}^{M}g(X_{m})f(X_{m})}{\sum_{m=1}^{M}g(X_{m})} - \frac{\mathbb{E}\left(f\left(X_1 \right) g\left(X_1\right)  \right)}{\E(g(X_1))}\right) } \prod_{m=1}^{M}\mathbb{E}\left(e^{\frac{iu}{\sqrt{M}}\left(f\left( \xi_m^{M}\right)-\mathbb{E}\left(f\left( \xi_m^{M}\right) \mathrel{|} {\cal F}  \right) \right)}\mathrel{\Big|}{\cal F}\right)\right). 
	 \end{align*}
         Using some Taylor expansion of $e^{\frac{iu}{\sqrt{M}}\left(f\left( \xi_m^{M}\right)-\mathbb{E}\left(f\left( \xi_m^{M}\right) \mathrel{|} {\cal F}  \right) \right)}$, one may approximate the product by a sum of integrals of functions of $$\left(\left\lbrace\bar S^M_{\lceil\alpha_1 M \rceil-1}\right\rbrace,\left\lbrace\bar S^M_{\lceil\alpha_2 M \rceil-1}\right\rbrace,\cdots,\left\lbrace\bar S^M_{\lceil\alpha_q M \rceil-1}\right\rbrace\right)$$ with respect to $1_{\{0<\alpha_1<\alpha_2<\cdots<\alpha_q<1\}}d\alpha_1d\alpha_2\cdots d\alpha_q$. This closely relates the asymptotic behaviour of the characteristic function to the one of the vectors $(\ref{purpose})$ for the choice  $Y_i=g(X_i)$, $\phi(x)=\frac 1 x$, $Z_i=g(X_i)f(X_i)$, $\beta^i_M=\lceil\alpha_i M \rceil-1$ for $i\in\{1,\cdots,q\}$ and $\beta^{q+1}_M=M$. In the companion paper \cite{echant}, using the main result of the present paper, we compute the asymptotic variance and prove the associated central limit theorem. 
	\section{Notation}\label{notation}
	We denote by $\mathbb{N}^*$ the set of natural numbers without $0$ and by $\mathbb{R}_+$  the set of non negative real numbers. 
	We denote by $\floor*{x}$ the integer $j$ such that $j\leq x < j+1$ and by $\{x\}=x-\floor*{x}$ the fractional part of $x\in \mathbb{R}$.\\ 
	We denote by $\mu_Y$ the law of a $\mathbb{R}^d$-valued random vector $Y=\left(Y^1,\cdots,Y^d\right)$ where $d\in \mathbb{N}^*$ and by $\phi_{Y}$ its characteristic function given by $\phi_{Y}(u)=\mathbb{E}(e^{iu^TY})=\int_{\mathbb{R}^d}e^{iu^Ty}\mu_{Y}(dy)$, $u\in\mathbb{R}^d$. The convergence in distribution is denoted by $\overset{d}{\Longrightarrow}$.
	Moreover let $ \mu_Y = \mu_{Y,c} + \mu_{Y,s}$ denote the decomposition of $\mu_Y$ into a part $\mu_{Y,c}$ absolutely continuous with respect to
	the Lebesgue measure and a singular part $\mu_{Y,s}$. Notice that there exists  $A$, a Borel subset of $\mathbb{R}^d$, such that $\mu_{Y,s}(A)=0$ and $\int_{\mathbb{R}^d}^{} 1_{\left\lbrace x\notin A\right\rbrace } dx = 0$. Let $p_Y$ denote a density of  $\mu_{Y,c}$ with respect to the Lebesgue measure. In what follows, we will always consider absolute continuity with respect to the Lebesgue measure and so we avoid to write it everytime. Moreover, when we write that the law of a random variable has an absolutely continuous component we mean a non zero absolutely continuous component.
	We denote by $m_Y=\left(\mathbb{E}\left(Y^1 \right),\cdots,\mathbb{E}\left(Y^d \right)  \right)$ the expected value vector of $Y$.\\
	For the total variation distance between the measures $\mu_1$ and $\mu_2$, we write

	\begin{align}
	d_{TV}\left(\mu_1,\mu_2 \right)&=\sup_{A\in \mathcal{B}\left(\mathbb{R}^d \right) }\left|\mu_1(A)-\mu_2(A) \right|=\frac{1}{2}\sup_{\left\|f\right\|_{\infty,\mathbb{R}}\leq 1}\left|\int_{\mathbb{R}^d}f(x)\left( \mu_1(dx)-\mu_2(dx) \right) \right|\label{dtvr}\\ 
	&= \frac{1}{2}\sup_{\left\|f\right\|_{\infty,\mathbb{C}}\leq 1}\left|\int_{\mathbb{R}^d}f(x)\left( \mu_1(dx)-\mu_2(dx) \right) \right| \label{dtvc}
	\end{align}
	where the first supremum is taken over the real-valued measurable functions and the second one is taken over the complex-valued measurable functions.
	The latter formulation is less usual so we will provide the proof of the third equality in the Appendix (see Lemma \ref{tvdistancecomplex}).\\
	Given two $\mathbb{R}^d$-valued random vectors $X$ and $Y$ and a measurable function $g:\mathbb{R}^d\rightarrow \mathbb{R}^{d'}$ where $d'\in \mathbb{N}^*$, the following inequality holds
	
	\begin{equation}\label{pushforward!}	d_{TV}(\mu_{g(X)},\mu_{g(Y)}) \leq d_{TV}(\mu_{X},\mu_{Y}).
	\end{equation}
	In particular,
	\begin{equation}\label{proptv}
	d_{TV}(\mu_{(0,X)},\mu_{(0,Y)}) = d_{TV}(\mu_{X},\mu_{Y})
	\end{equation}
	and given $A\in\mathbb{R}^{d\times d}$ an invertible matrix
	\begin{equation}\label{proptv2}
	d_{TV}(\mu_{AX},\mu_{AY}) = d_{TV}(\mu_{X},\mu_{Y}).
	\end{equation}
	Additionally we introduce the following notation: 
	given a real sequence $(x_j)_{j\geq 1}$ and a real number $c\in \mathbb{R}$ we define for any integer $M\geq 1$ 
	$$\bar{x}_M^c:= \frac{x_1+\cdots+x_M+c}{M} .$$ 
	
	\section{Main Result}
	
	Let $\left(Y_i,Z_i\right)_{i\geq 1}$ be a sequence of square-integrable i.i.d. random vectors in $\mathbb{R}^2$ where $Y_i$ is not constant. Moreover let $q\in \mathbb{N}^*$ and consider a sequence of vectors of integers $(\beta_M^1,\cdots,\beta_M^{q+1} )_{M\geq 1}\subseteq \mathbb{N}^{q+1}$ such that  $$\lim\limits_{M\rightarrow \infty}\sqrt{M}\left( \frac{\beta_M^i}{M}- \beta^i\right) =0$$  with $$0<\beta^1< \cdots <\beta^{q+1}.$$
	Let us observe that for $M$ big enough $0<\beta_M^1<\cdots<\beta_M^{q+1}$.
	Given $(x,z,y_1,\cdots,y_q)\in \mathbb{R}^{q+2}$ and $\phi$  a  measurable real-valued function, we are interested in studying the convergence in distribution as $M\rightarrow\infty$ of the following random vector
	
	\begin{equation}\label{vector_to_study2}
	\left( \left\lbrace R_{\beta_M}^{1,y_1}\right\rbrace ,\cdots,\left\lbrace R_{\beta_M}^{q,y_q}\right\rbrace ,\sqrt{M}\left(\phi\left( \frac{\beta^{q+1}_M}{M}\overline{Y}_{\beta_M^{q+1}}^{x}\right)\times\frac{\beta^{q+1}_M}{M} \overline{Z}_{\beta_M^{q+1}}^z- \theta\right)\right) 
	\end{equation}
	where
	\begin{equation} \label{R}
	R_{\beta_M}^{i,y_i}:= \phi\left(\frac{\beta_M^{q+1}}{M}\overline{Y}_{\beta_M^{q+1}}^{x}\right) \left( Y_1+\cdots+Y_{\beta_M^i}+y_i\right)
	\end{equation}

	and  \begin{equation} \label{theta}\theta= \phi\left( \beta^{q+1}m_Y\right)\beta^{q+1}m_Z.
	\end{equation}
	
	\begin{remark}\label{deltamethod}
		Let us observe that  if $\phi$ is differentiable at $\beta^{q+1}m_Y$, the application of the delta method provides the convergence in distribution of the last component of \eqref{vector_to_study2}: 
		\begin{align}\label{tcl}
		\sqrt{M}\left(\phi\left( \frac{\beta^{q+1}_M}{M}\overline{Y}_{\beta_M^{q+1}}^{x}\right)\times\frac{\beta^{q+1}_M}{M} \overline{Z}_{\beta_M^{q+1}}^z- \theta\right) \overset{d}{\Longrightarrow} T \sim \mathcal{N}\left(0,\sigma^2_{T} \right)
		\end{align}
		where $\sigma^2_{T}=\beta^{q+1} \left(\phi'(\beta^ {q+1}m_Y)\beta^ {q+1}m_Z,\phi(\beta^ {q+1}m_Y)\right)^T \Sigma_{(Y_1,Z_1)}\left(\phi'(\beta^ {q+1}m_Y)\beta^ {q+1}m_Z,\phi(\beta^ {q+1}m_Y)\right)$ being $\Sigma_{(Y_1,Z_1)}$ the covariance matrix of $(Y_1,Z_1)$.
	\end{remark}

	We are now ready to state the main result of this work.
	
	\begin{teo} \label{mainresult}
		Let $\left(Y_i,Z_i\right)_{i\geq 1}$ be a sequence of square-integrable i.i.d. random vectors in $\mathbb{R}^2$ such that the law of $Y_i$ has an absolutely continuous component.
		Moreover let  $(x,z,y_1,\cdots,y_q)\in \mathbb{R}^{q+2}$ and let $\phi:\mathbb{R}\rightarrow \mathbb{R}$ be a measurable function  differentiable at $\beta^{q+1}m_{Y}$  such that $\phi(\beta^{q+1}m_Y)\neq 0$. 
		If there exists $ \tilde{M}\in \mathbb{N}$ such that 
		\begin{equation}\label{dominateconvergence1}	
		\int_{\mathbb{R}}\frac{1}{\left|\phi(\beta^{q+1}m_Y+y) \right|^{q+1}}e^{-\tilde{M}  y^2}dy < \infty,
		\end{equation}
		then the random vector 
		
		$$
		\left(\left\lbrace R_{\beta_M}^{1,y_1}\right\rbrace ,\cdots,\left\lbrace R_{\beta_M}^{q,y_q} \right\rbrace,\sqrt{M}\left(\phi\left(\frac{\beta^{q+1}_M}{M} \overline{Y}_{\beta_M^{q+1}}^{x}\right) \times \frac{\beta^{q+1}_M}{M} \overline{Z}_{\beta_M^{q+1}}^z- \theta\right)\right)$$ 
		converges in distribution as $M\rightarrow \infty$ to $\left(U,T \right) $ where $T$ has been introduced in \eqref{tcl} and where $U$ is a uniform random variable on $\left[ 0,1\right]^{q}$ independent of $T$.
		
	\end{teo}
	\begin{remark} \label{notcl}
		Let us observe that we can apply Theorem \ref{mainresult} to $\left(Y_i,0 \right)_{i\geq 1}$ where $\left(Y_i \right)_{i\geq 1}$ is a sequence of square-integrable i.i.d. real-valued random variables such that the law of $Y_i$ has an absolutely continuous component. In particular we have
		$$
		\left(\left\lbrace R_{\beta_M}^{1,y_1}\right\rbrace ,\cdots,\left\lbrace R_{\beta_M}^{q,y_q} \right\rbrace\right) \overset{d}{\Longrightarrow} U$$ 
		where $U$ is a uniform random variable on $\left[ 0,1\right]^{q}$. It is possible to prove that in this case the hypothesis \eqref{dominateconvergence1} can be replaced by the following slightly weaker hypothesis: $ \exists\tilde{M}\in \mathbb{N}$ such that 
		\begin{equation}	
		\int_{\mathbb{R}}\frac{1}{\left|\phi(\beta^{q+1}m_Y+y) \right|^{q}}e^{-\tilde{M}  y^2}dy < \infty.
		\end{equation}
	\end{remark}

	We are now going to provide the statement of Theorem \ref{mainresult} in the particular case where the law of $\left(Y_1,Z_1\right)$  has an absolutely continuous component and $(z,y_1,\cdots,y_q)= \left(0,\cdots,0 \right) $.  This, together with the lemma that immediately follows, will allow to prove Theorem \ref{mainresult}. 
	
	\begin{prop}\label{tcl_general}
		Theorem \ref{mainresult} holds under the reinforced hypotheses that the law of $\left(Y_i,Z_i\right)$ has an absolutely continuous component and $(z,y_1,\cdots,y_q)= \left(0,\cdots,0 \right) $. 
	\end{prop}
	
	\begin{lemma}\label{abs_continuity}
		Let $Y$ and $Z$ be two real-valued random variables such that the law of $Y$ has an absolutely continuous component. If $\xi$ is an absolutely continuous real-valued random variable independent of $\left(Y,Z \right)$, the law of $\left(Y,Z+ \xi\right) $ has an absolutely continuous component.
	\end{lemma}
	We provide the proof of Proposition \ref{tcl_general} in Section \ref{prooflemma} and the proof of Lemma \ref{abs_continuity} in the Appendix.
	
	\begin{proof}[Proof of Theorem \ref{mainresult}] The proof consists of two steps. In step \textit{(i)} we are going to check that the conclusion still holds when the hypothesis of existence of an absolutely continuous component for $\left(Y_1,Z_1 \right)$ made in Proposition \ref{tcl_general} is weakened to the existence of an absolutely continuous component for $Y_1$. Moreover we suppose that  $(z,y_1,\cdots,y_q)= \left(0,\cdots,0 \right) $. In step \textit{(ii)} we deal with the case when the vector $(z,y_1,\cdots,y_q)\neq \left(0,\cdots,0 \right) $ . \\ \textit{(i)} 
		To simplify the notation in what follows we write $ R_{\beta_M}^{i}$ instead of $ R_{\beta_M}^{i,0}$ and $\overline{Z}_{\beta_M^{q+1}}$ instead of $\overline{Z}_{\beta_M^{q+1}}^0$.\\  
		Let $\left(\xi_i \right)_{i\geq 1}$ be a sequence of zero-mean absolutely continuous square-integrable i.i.d. real-valued random variables independent of $\left(Y_i,Z_i \right)_{i\geq 1} $. For each $n\geq 1$, let $\xi_i^n=\frac{\xi_i}{n}$ and let us consider the sequence $\left(Y_i,Z_i+\xi_i^n \right)_{i\geq 1}$ of square-integrable i.i.d. random vectors in $\mathbb{R}^2$ such that by Lemma \ref{abs_continuity} the law of $\left(Y_i,Z_i+\xi_i^n \right)$ has an absolutely continuous component.  
		Thus we can apply Proposition \ref{tcl_general} and obtain that for each $n\geq 1$ the random vector 
		
		\begin{align}\label{conv-dis}
		\left(\left\lbrace R_{\beta_M}^{1}\right\rbrace ,\cdots,\left\lbrace R_{\beta_M}^{q} \right\rbrace,\sqrt{M}\left(\phi\left(\frac{\beta^{q+1}_M}{M} \overline{Y}_{\beta_M^{q+1}}^{x}\right) \times \frac{\beta^{q+1}_M}{M} \left( \overline{Z}_{\beta_M^{q+1}}+\overline{\xi^n}_{\beta_M^{q+1}}\right) - \theta\right)\right) \overset{d}{\Longrightarrow} \left(U,T_{n} \right)
		\end{align} 
		
		where $\theta=\phi\left( \beta^{q+1}m_Y\right)\beta^{q+1}m_Z$, $U$ is a uniform random variable on $\left[ 0,1\right]^{q}$ independent of $T_{n}$ and by Remark \ref{deltamethod}, $T_{n}\sim \mathcal{N}\left(0,\sigma^2_{T_{n}} \right)$ with 
		\begin{align*}
		\sigma^2_{T_{n}}&= \beta^{q+1} \left(\phi'(\beta^ {q+1}m_Y)\beta^ {q+1}m_Z,\phi(\beta^ {q+1}m_Y)\right)^T \Sigma_{(Y_1,Z_1+\xi_1^n)}\left(\phi'(\beta^ {q+1}m_Y)\beta^ {q+1}m_Z,\phi(\beta^ {q+1}m_Y)\right).
		\end{align*}
		Since $(Y_1,Z_1+\xi_1^n)$ converges in $L^2$ to $(Y_1,Z_1)$ as $n$ goes to infinity, $\Sigma_{(Y_1,Z_1+\xi_1^n)}\underset{n\rightarrow \infty}{\longrightarrow}\Sigma_{(Y_1,Z_1)}$ and $\sigma^2_{T_{n}}\underset{n\rightarrow \infty}{\longrightarrow} \sigma^2_{T} = \beta^{q+1} \left(\phi'(\beta^ {q+1}m_Y)\beta^ {q+1}m_Z,\phi(\beta^ {q+1}m_Y)\right)^T \Sigma_{(Y_1,Z_1)}\left(\phi'(\beta^ {q+1}m_Y)\beta^ {q+1}m_Z,\phi(\beta^ {q+1}m_Y)\right).$
		\\ Let us now prove that this implies the convergence in distribution as $M\rightarrow \infty$ of (\ref{vector_to_study2}) to $\left(U,T \right)$ where $T\sim \mathcal{N}\left(0, \sigma^2_{T}\right)  $ independent of $U= \left(U_1,\cdots,U_q \right) $.\\
		Let $\left(u_1,u_2 \right)\in \mathbb{R}^{q}\times  \mathbb{R}$ and let us denote the random vector $\left(\left\lbrace R_{\beta_M}^{1}\right\rbrace,\cdots, \left\lbrace R_{\beta_M}^{q}\right\rbrace\right)$ by $\left\lbrace R_{\beta_M}^{i}\right\rbrace_{i\leq q}$. For $ n\geq 1$  one has 
		
		\begin{align*}
		&\left| \mathbb{E}\left( e^{iu_1^T\left(\left\lbrace R_{\beta_M}^{i}\right\rbrace_{i \leq q} \right) +iu_2\sqrt{M}\left(\phi\left(\frac{\beta^{q+1}_M}{M} \overline{Y}_{\beta_M^{q+1}}^{x}\right) \times \frac{\beta^{q+1}_M}{M}  \overline{Z}_{\beta_M^{q+1}} - \theta\right) }\right)-\mathbb{E}\left(e^{iu_1^TU} \right) e^{-\frac{1}{2}\sigma^2_{T}u_2^2} \right|\\
		& \leq  \left| \mathbb{E}\left( e^{iu_1^T\left(\left\lbrace R_{\beta_M}^{i}\right\rbrace_{ i \leq q} \right)}e^{iu_2\sqrt{M}\left(\phi\left(\frac{\beta^{q+1}_M}{M} \overline{Y}_{\beta_M^{q+1}}^{x}\right) \times \frac{\beta^{q+1}_M}{M}  \overline{Z}_{\beta_M^{q+1}} - \theta\right) }\left(1-e^{iu_2\sqrt{M}\phi\left(\frac{\beta^{q+1}_M}{M} \overline{Y}_{\beta_M^{q+1}}^{x}\right) \frac{\beta^{q+1}_M}{M}  \overline{\xi^n}_{\beta_M^{q+1}} } \right) \right)  \right| \\
		&\phantom{==}+\left| \mathbb{E}\left( e^{iu_1^T\left(\left\lbrace R_{\beta_M}^{i}\right\rbrace_{ i \leq q} \right) +iu_2\sqrt{M}\left(\phi\left(\frac{\beta^{q+1}_M}{M} \overline{Y}_{\beta_M^{q+1}}^{x}\right) \times \frac{\beta^{q+1}_M}{M} \left( \overline{Z}_{\beta_M^{q+1}}+\overline{\xi^n}_{\beta_M^{q+1}}\right) - \theta\right) }\right)-\mathbb{E}\left(e^{iu_1^TU} \right) e^{-\frac{1}{2}\sigma^2_{T_{n}}u_2^2} \right|\\
		&\phantom{==}+\left|\mathbb{E}\left(e^{iu_1^TU} \right) e^{-\frac{1}{2}\sigma^2_{T_{n}}u_2^2}-\mathbb{E}\left(e^{iu_1^TU} \right) e^{-\frac{1}{2}\sigma^2_{T}u_2^2} \right|\\
		&\leq \mathbb{E}\left( \left|1- e^{iu_2\sqrt{M}\phi\left(\frac{\beta^{q+1}_M}{M} \overline{Y}_{\beta_M^{q+1}}^{x}\right) \frac{\beta^{q+1}_M}{M}  \overline{\xi^n}_{\beta_M^{q+1}} } \right| \right)  \\
		&\phantom{==}+\left| \mathbb{E}\left( e^{iu_1^T\left(\left\lbrace R_{\beta_M}^i\right\rbrace_{ i \leq q} \right) +iu_2\sqrt{M}\left(\phi\left(\frac{\beta^{q+1}_M}{M} \overline{Y}_{\beta_M^{q+1}}^{x}\right) \times \frac{\beta^{q+1}_M}{M} \left( \overline{Z}_{\beta_M^{q+1}}+\overline{\xi^n}_{\beta_M^{q+1}}\right) - \theta\right) }\right)-\mathbb{E}\left(e^{iu_1^TU} \right) e^{-\frac{1}{2}\sigma^2_{T_{n}}u_2^2} \right|\\
		&\phantom{==}+\left| e^{-\frac{1}{2}\sigma^2_{T_{n}}u_2^2}- e^{-\frac{1}{2}\sigma^2_{T}u_2^2}\right| .
		\end{align*}
		
		Let us now study the first term of the right-hand side.\\ By observing that $H_M=\sqrt{M}\phi\left(\frac{\beta^{q+1}_M}{M} \overline{Y}_{\beta_M^{q+1}}^{x}\right) \frac{\beta^{q+1}_M}{M}  \overline{\xi}_{\beta_M^{q+1}} \overset{d}{\Longrightarrow} \mathcal{N}\left(0, \beta^{q+1}\phi^2\left(\beta^{q+1}m_Y \right) \textrm{Var}\left( \xi\right) \right)$, this in particular implies that the sequence $H_M$ is tight:  $\forall \epsilon >0$ there exists $K_\epsilon>0$ such that $\sup_{M\geq 1}\mathbb{P}\left(\left| H_M\right|>  K_\epsilon  \right)\leq \epsilon$. One has

		\begin{align*}
		\mathbb{E}&\left( \left|1- e^{iu_2\sqrt{M}\phi\left(\frac{\beta^{q+1}_M}{M} \overline{Y}_{\beta_M^{q+1}}^{x}\right) \frac{\beta^{q+1}_M}{M}  \overline{\xi^n}_{\beta_M^{q+1}} } \right| \right)=\mathbb{E}\left( \left|1- e^{i\frac{u_2}{n} H_M } \right| \right)\\
		&\phantom{=}=\mathbb{E}\left( \left|1- e^{i\frac{u_2}{n} H_M } \right|1_{\left|H_M \right| \leq K_\epsilon } \right)+\mathbb{E}\left( \left|1- e^{i\frac{u_2}{n} H_M } \right|1_{\left|H_M \right| > K_\epsilon } \right)\\
		&\phantom{=}\leq \frac{\left|u_2 \right| }{n}K_\epsilon+2\epsilon
		\end{align*}
		
		where to obtain the last inequality we use that $\left|1-e^{ix} \right|\leq \left| x\right|$,  $\forall x\in \mathbb{R}$.

		Therefore we have obtained that $\forall \epsilon>0$, $\forall n>0$
		\begin{align*}
		&\left| \mathbb{E}\left( e^{iu_1^T\left(\left\lbrace R_{\beta_M}^i\right\rbrace_{ i \leq q} \right) +iu_2\sqrt{M}\left(\phi\left(\frac{\beta^{q+1}_M}{M} \overline{Y}_{\beta_M^{q+1}}^{x}\right) \times \frac{\beta^{q+1}_M}{M} \overline{Z}_{\beta_M^{q+1}} - \theta\right) }\right)-\mathbb{E}\left(e^{iu_1^TU} \right) e^{-\frac{1}{2}\sigma^2_{T}u_2^2} \right|\\
		&\phantom{=}\leq \frac{\left|u_2 \right| }{n}K_\epsilon+2\epsilon  + \left|  e^{-\frac{1}{2}\sigma^2_{T_{n}}u_2^2}-e^{-\frac{1}{2}\sigma^2_{T}u_2^2} \right| \\
		&\phantom{==}+ \left| \mathbb{E}\left( e^{iu_1^T\left(\left\lbrace R_{\beta_M}^i\right\rbrace_{ i \leq q} \right) +iu_2\sqrt{M}\left(\phi\left(\frac{\beta^{q+1}_M}{M} \overline{Y}_{\beta_M^{q+1}}^{x}\right) \times \frac{\beta^{q+1}_M}{M} \left( \overline{Z}_{\beta_M^{q+1}}+\overline{\xi^n}_{\beta_M^{q+1}}\right) - \theta\right) }\right)-\mathbb{E}\left(e^{iu_1^TU} \right) e^{-\frac{1}{2}\sigma^2_{T_{n}}u_2^2} \right|.
		\end{align*}
		To conclude this part of the proof it is sufficient to take in the above expression first the superior limit as $M\rightarrow \infty$ so that the last term converges to $0$ by (\ref{conv-dis}), then the superior limit as $n\rightarrow \infty$ so that $\frac{\left|u_2 \right| }{n}K_\epsilon\rightarrow 0$ and $\left| e^{-\frac{1}{2}\sigma^2_{T}u_2^2}- e^{-\frac{1}{2}\sigma^2_{T_{n}}u_2^2} \right|\rightarrow 0$ and finally the superior limit as $\epsilon\rightarrow 0$.\\
		
		\textit{(ii)} Let us now take $(z,y_1,\cdots,y_q)\in \mathbb{R}^{q+2}$ different from the zero vector.	One has
		
		\begin{align}
		&\left(\left\lbrace R_{\beta_M}^{1,y_1}\right\rbrace ,\cdots,\left\lbrace R_{\beta_M}^{1,y_q} \right\rbrace,\sqrt{M}\left(\phi\left(\tfrac{\beta^{q+1}_M}{M} \overline{Y}_{\beta_M^{q+1}}^{x}\right) \times \tfrac{\beta^{q+1}_M}{M} \overline{Z}_{\beta_M^{q+1}}^z- \theta\right)\right)\label{conv_vettore} \\ \nonumber
		& = \left( \left( \left\lbrace     
		\phi\left(\tfrac{\beta_M^{q+1}}{M}\overline{Y}_{\beta_M^{q+1}}^{x}\right)y_i+
		\left\lbrace  R_{\beta_M}^{i}\right\rbrace\right\rbrace\right)_{1\leq i \leq q}  ,\sqrt{M}\left(\phi\left(\tfrac{\beta^{q+1}_M}{M} \overline{Y}_{\beta_M^{q+1}}^{x}\right) \times \tfrac{\beta^{q+1}_M}{M} \overline{Z}_{\beta_M^{q+1}}- \theta\right)\right)\\\nonumber
		&\phantom{==}+\left(0,\cdots,0,\frac{z}{\sqrt{M}}\phi\left(\tfrac{\beta^{q+1}_M}{M} \overline{Y}_{\beta_M^{q+1}}^{x}\right) \right) .\nonumber
		\end{align}
		
		By the Strong Law of Large Numbers and the hypothesis of continuity of $\phi$ at $\beta^{q+1}m_Y$,  $\lim\limits_{M\rightarrow \infty}\phi\left(\tfrac{\beta_M^{q+1}}{M}\overline{Y}_{\beta_M^{q+1}}^{x}\right)y_i=\phi\left(\beta^{q+1}m_Y\right)y_i$. We can therefore apply the previous step and Slutsky's theorem to deduce the following convergence in distribution 
		\begin{align*}
		&\left(      
		\phi\left(\tfrac{\beta_M^{q+1}}{M}\overline{Y}_{\beta_M^{q+1}}^{x_2}\right)y_1+
		\left\lbrace  R_{\beta_M}^1\right\rbrace  ,\cdots, \phi\left(\tfrac{\beta_M^{q+1}}{M}\overline{Y}_{\beta_M^{q+1}}^{x}\right)y_q+ \left\lbrace R_{\beta_M}^q \right\rbrace ,\sqrt{M}\left(\phi\left(\tfrac{\beta^{q+1}_M}{M} \overline{Y}_{\beta_M^{q+1}}^{x}\right) \times \tfrac{\beta^{q+1}_M}{M} \overline{Z}_{\beta_M^{q+1}}- \theta\right)\right)\\
		&\phantom{===} \overset{d}{\Longrightarrow}\left(\phi\left(\beta^{q+1}m_Y\right)y_1+U_1,\cdots, \phi\left(\beta^{q+1}m_Y\right)y_q+U_q,T\right). 
		\end{align*}
		
		Now by observing that the set of the points of discontinuity of the function $(x_1,\cdots,x_q,y)\longmapsto \left(\left\lbrace x_1 \right\rbrace ,\cdots ,\left\lbrace x_q \right\rbrace,y \right)$ has a zero measure with respect to the law of $\left(\phi\left(\beta^{q+1}m_Y\right)y_1+U_1,\cdots, \phi\left(\beta^{q+1}m_Y\right)y_q+U_q,T\right)$ and applying the continuous mapping theorem we can deduce that 
		
		\begin{align*}
		& \left( \left( \left\lbrace     
		\phi\left(\tfrac{\beta_M^{q+1}}{M}\overline{Y}_{\beta_M^{q+1}}^{x}\right)y_i+
		\left\lbrace  R_{\beta_M}^i\right\rbrace\right\rbrace\right)_{1\leq i \leq q}  ,\sqrt{M}\left(\phi\left(\tfrac{\beta^{q+1}_M}{M} \overline{Y}_{\beta_M^{q+1}}^{x}\right) \times \tfrac{\beta^{q+1}_M}{M} \overline{Z}_{\beta_M^{q+1}}- \theta\right)\right)\\
		& \phantom{===} \overset{d}{\Longrightarrow}\left(\left\lbrace \phi\left(\beta^{q+1}m_Y\right)y_1+U_1\right\rbrace ,\cdots, \left\lbrace \phi\left(\beta^{q+1}m_Y\right)y_q+U_q\right\rbrace ,T\right) \overset{d}{=} \left(U_1,\cdots ,U_q,T \right). 
		\end{align*}
		This combined with the fact that $\left(0,\cdots,0,\frac{z}{\sqrt{M}}\phi\left(\tfrac{\beta^{q+1}_M}{M} \overline{Y}_{\beta_M^{q+1}}^{x}\right) \right) $ converges as $M$ goes to infinity to the zero vector allows to conclude that 
		(\ref{conv_vettore}) converges in distribution to $(U,T)$.
		
	\end{proof}
	
	\section{Proof of Proposition \ref{tcl_general}}\label{prooflemma}
	
	Before proving Proposition \ref{tcl_general}, we need some preliminary results. As done in the proof of Theorem \ref{mainresult}, in what follows we write $ R_{\beta_M}^{i}$ instead of $ R_{\beta_M}^{i,0}$ and $\overline{Z}_{\beta_M^{q+1}}$ instead of $\overline{Z}_{\beta_M^{q+1}}^0$.
	Let us observe that for each $j=1,\cdots,q$  
	\begin{align*}\label{vector_to_study_0}
	R_{\beta_M}^j 
	&=\sum\limits_{\ell=1}^{j}R_{\beta_M}^{\ell-1:\ell}
	\end{align*}
	where we have introduced the notation $R_{\beta_M}^{\ell-1:\ell}:=\phi\left(\frac{\beta_M^{q+1}}{M}\overline{Y}_{\beta_M^{q+1}}^{x}\right) \left(Y_{\beta_M^{\ell - 1}+1}+\cdots+Y_{\beta_M^{\ell }}\right)$ with $\beta_M^0:=0$ for each $M$ by convention. 
	Let us therefore study the asymptotic behaviour of the vector
	
	\begin{equation}\label{Rnbeta}
	\left(R_{M} ,K_M\right)
	\end{equation}
	where $$R_{M} :=\left(R_{\beta_M}^{{0:1}},R_{\beta_M}^{1:2},\cdots,R_{\beta_M}^{q-1:q} \right)$$
	and 
	$$ K_M:= \sqrt{M}\left(\phi\left( \frac{\beta_M^{q+1}}{M}\overline{Y}_{\beta_M^{q+1}}^{x}\right)\times \frac{\beta^{q+1}_M}{M} \overline{Z}_{\beta_M^{q+1}} - \theta\right)=\frac{\phi\left( \frac{\beta_M^{q+1}}{M}\overline{Y}_{\beta_M^{q+1}}^{x}\right)\beta^{q+1}_M\overline{Z}_{\beta_M^{q+1}} -M\theta}{\sqrt{M}} $$
	with $\theta$ as in (\ref{theta}).\\

	The proof of the following proposition is given in Section \ref{sectionproofofprop}.

	\begin{prop}\label{asymptoticRn} 
		Let $\gamma= \left( \beta^1,\beta^2-\beta^1,\cdots, \beta^{q}-\beta^{q-1}\right)$.  Under the assumptions of Proposition \ref{tcl_general} the following convergence in total variation holds
		\begin{equation}\label{caseR2}
		d_{TV}\bigg(\mu_{\left(  \frac{R_{M}- m_{Y} \phi(\beta^{q+1}m_{Y})M\gamma }{\sqrt{M}},K_M \right)},\mathcal{N}\left( 0,\Gamma\right) \bigg) \underset{M\rightarrow \infty}{\longrightarrow} 0
		\end{equation}
		where $\Gamma\in \mathbb{R}^{(q+1)\times (q+1)}$ is a positive definite covariance matrix. 		
	\end{prop}
	
	Let us now state the Weyl criterion  concerning the convergence in distribution to a vector composed of a uniform random variable on $\left[ 0,1\right]^q$ and an independent vector. 
	
	\begin{teo}[Weyl criterion]\label{weyl.criterion}
		Let $\left( B_M\right)_{M\geq 1} $
		be a sequence of  $\mathbb{R}^q$-valued random vectors and let $\left( H_M\right)_{M\geq 1} $ be a sequence of    $\mathbb{R}^{q'}$-valued random vectors that converges in distribution to a $\mathbb{R}^{q'}$-valued random vector $H$. Then as $M\rightarrow \infty$ the sequence $\left( \left\lbrace B_{M}^1\right\rbrace ,\cdots,\left\lbrace B_{M}^q \right\rbrace,H_M\right)$ converges in distribution   to $\left(U,H \right) $ where $U$ is independent of $H$ and uniformly distributed on $\left[ 0,1\right]^q$ if and only if for every $k\in \mathbb{Z}^q\backslash \left\lbrace 0 \right\rbrace $  and $u\in \mathbb{R}^{q'}$
		
		\begin{equation} \label{suffcond}
		\lim_{M\rightarrow \infty}\phi_{\left( B_{M},H_M\right) }\left(2k\pi ,u\right) =0.
		\end{equation}
	\end{teo}
	
	For the sake of completeness, the proof of Theorem \ref{weyl.criterion} is provided in Section \ref{sectionproofofprop}. 	We are now ready to prove Proposition
	\ref{tcl_general}.
	
	\begin{proof}[Proof of Proposition \ref{tcl_general}] 
		By Theorem \ref{weyl.criterion}, it is sufficient to prove that for each $k\in \mathbb{Z}^q\backslash \left\lbrace 0 \right\rbrace $ , $u\in \mathbb{R}$
		\begin{align*}
		\lim_{M\rightarrow \infty}\phi_{\left( R_{\beta_M}^1 ,\cdots, R_{\beta_M}^q ,\sqrt{M}\left(\phi\left(\frac{\beta^{q+1}_M}{M} \overline{Y}_{\beta_M^{q+1}}^{x}\right) \times \frac{\beta^{q+1}_M}{M} \overline{Z}_{\beta_M^{q+1}}- \theta\right)   \right) }(2k\pi,u)=0.
		\end{align*}

		Let us observe that it is equivalent to prove that for each $k\in \mathbb{Z}^{q}\backslash \left\lbrace 0 \right\rbrace $ , $u\in \mathbb{R}$
		\begin{equation}\label{pur2}
		\lim_{M\rightarrow \infty}\phi_{\left(R_M,K_M \right) }(2k\pi,u)=0
		\end{equation}
		where $\left(R_M,K_M \right)$ has been introduced in (\ref{Rnbeta}). 
		
		By Proposition \ref{asymptoticRn}, 
		\begin{equation}\label{genres}
		\lim_{M\rightarrow \infty}d_{TV}(\mu_{\left( a_MR_{M}+b_M,K_M\right) }, \mu_{Y})=0
		\end{equation} 
		where $a_M=\frac{1}{\sqrt{M}}$, $b_M = - m_{Y} \phi(\beta^{q+1}m_{Y})\sqrt{M}\gamma$ and $Y\sim \mathcal{N}(0,\Gamma)$, $\Gamma$ positive definite.
		
		For $u:=(u_1,u_2)\in \left\lbrace \mathbb{R}^{q}\backslash \left\lbrace 0 \right\rbrace\right\rbrace \times\mathbb{R} $
		\begin{align*}
		\left| \phi_{\left( R_{M},K_M\right) }(u_1,u_2)\right|&= \left| \phi_{\left( a_MR_{M}+b_M,K_M\right)}\left( \frac{u_1}{a_M},u_2\right)e^{-\frac{iu_1^Tb_M}{a_M}} \right|= \left| \phi_{\left( a_MR_{M}+b_M,K_M\right)}\left( \frac{u_1}{a_M},u_2\right)\right|\\
		&\leq \left| \phi_{\left( a_MR_{M}+b_M,K_M\right)}\left( \frac{u_1}{a_M},u_2\right) -\phi_{Y}\left( \frac{u_1}{a_M},u_2\right) \right|+\left|\phi_{Y}\left( \frac{u_1}{a_M},u_2\right) \right|.
		\end{align*}
		Since by (\ref{dtvc}) the first term of the right-hand side can be bounded from above by  $2d_{TV}(\mu_{\left( a_MR_{M}+b_M,K_M\right) }, \mu_{Y}),$
		we deduce that  $$\limsup_{M\rightarrow \infty}\left| \phi_{\left( R_{M},K_M\right)}(u_1,u_2)\right|\leq \limsup_{M\rightarrow \infty}\left|\phi_{Y}\left( \frac{u_1}{a_M},u_2\right) \right|.$$
		Since the law of $Y$ is absolutely continuous and $\lim_{M\rightarrow \infty}a_M=0$, the right-hand side goes to $0$ as $M\rightarrow \infty$ by the Riemann Lebesgue lemma. In particular (\ref{pur2}) is true. \\
	\end{proof}
	\section{Proof of Proposition \ref{asymptoticRn} and Theorem \ref{weyl.criterion} } \label{sectionproofofprop}
	In this section we are first going to prove Proposition \ref{asymptoticRn}. Let $\left(Y_i,Z_i\right)_{i\geq 1}$ be a sequence of square-integrable i.i.d. random vectors in $\mathbb{R}^2$ such that the law of $\left(Y_i,Z_i\right)$ has an absolutely continuous component. We denote by  $\Sigma$ its covariance matrix which has rank $2$.
	
	The proof of Proposition \ref{asymptoticRn} strongly relies on the following result. 
	\begin{teo}\label{proho_generale}
		Let $\left(Y_i,Z_i\right)_{i\geq 1}$ be a sequence of square-integrable i.i.d. random vectors in $\mathbb{R}^2$ such that the law of $\left(Y_i,Z_i\right)$ has an absolutely continuous component. Under the notations introduced above
		\begin{align*}
		\lim_{M\rightarrow \infty}d_{TV}\bigg(\mu_{\frac{1}{\sqrt{M}}\sum\limits_{k=1}^{M}\left(
			Y_{k}-m_Y,Z_{k}- m_Z
			\right)  }, \mu_{G} \bigg)=0
		\end{align*} 
		where $G \sim \mathcal{N}(0,\Sigma)$.
	\end{teo} 
	
	Let $F$ be a centered square-integrable random variable in $\mathbb{R}^n$ with identity covariance matrix and let $F_k$, $k\in \mathbb{N}^*$, independent copies of $F$. 
	The main instrument that we will use in the proof of Theorem  \ref{proho_generale} is the result about the convergence in total variation for the CLT that is $\lim_{M\rightarrow \infty}d_{TV}\bigg(\mu_{\frac{1}{\sqrt{M}}\sum_{k=1}^{M}F_k},\mathcal{N}(0, I_{n\times n}) \bigg)=0$ where $I_{n\times n}$ denotes the identity matrix of size $n$.\\
	Prohorov \cite{prohorov} in $1952$ was the first to give his contribution to the problem: he proved that, in dimension $1$,  a necessary and sufficient
	condition in order to get the result is that there exists $M_0$ such that the law of $\sum_{k=1}^{M_0}F_k$ has an
	absolutely continuous component. Recently Bally and Caramellino \cite{ballycaramellino} contribute in the same direction by extending the result to any dimension. They prove that under the assumption that the law of $F$ has an absolutely continuous component, $	\lim_{M\rightarrow \infty}d_{TV}\bigg(\mu_{\frac{1}{\sqrt{M}}\sum\limits_{k=1}^{M}F_k},\mathcal{N}(0, I_{n\times n}) \bigg)=0$.

	We are now ready to prove Theorem \ref{proho_generale}.
	\begin{proof}[Proof of Theorem \ref{proho_generale}]
		Let $O\in \mathbb{R}^{2\times 2}$ be an orthogonal matrix $\left( O^{T}O=I\right)$  that diagonalizes $\Sigma$ that is $\Sigma = O^TDO$ where $D$ is the diagonal matrix containing the eigenvalues $\lambda_1, \lambda_2$ of  $\Sigma$.	
		Given $G\sim \mathcal{N}(0,\Sigma)$ and by introducing the following notation $D^{-\frac{1}{2}}= \mathrm{diag}\left(\frac{1}{\sqrt{\lambda_1}} ,\frac{1}{\sqrt{\lambda_2}} \right)$, we have 
		\begin{align*} 
		d_{TV}\bigg(\mu_{\frac{1}{\sqrt{M}}\sum\limits_{k=1}^{M}\left(Y_{k}-m_Y,Z_{k} -m_Z\right)}, \mu_{G} \bigg)=d_{TV}\left(\mu_{\frac{D^{-\frac{1}{2}}}{\sqrt{M}}\sum\limits_{k=1}^{M}O\left(Y_{k}-m_Y,Z_{k}-m_Z\right)},\mu_{D^{-\frac{1}{2}}OG} \right)
		\end{align*}
		
		where to obtain the above equality we apply (\ref{proptv2}) with $A=D^{-\frac{1}{2}}O$. By the hypothesis that the law of $\left( Y_1,Z_1\right)$ has an absolutely continuous component, the right-hand side converges to $0$ as $M\rightarrow \infty$ by the result of Bally and Caramellino mentioned above.
	\end{proof}	
	
	We now recall a result obtained by Parthasarathy and Steerneman in the second section of \cite{Steerneman} regarding the behavior of the total variation convergence with respect to the sum and the multiplication by a real sequence.
	
	\begin{lemma}\label{parthasarathy}
		Let $\left( B_M\right)_{M\geq1}$  and $\left( T_M\right)_{M\geq1}$ be two independent sequences of  $\mathbb{R}^d$-valued random variables  such that $\lim_{M\rightarrow \infty}d_{TV}\left(\mu_{B_M},\mu_{B}\right)=0$ and $T_M\overset{d}{\Longrightarrow }  T$ for  $\mathbb{R}^d$-valued random variables $B,T$. If $\mu_{B}$ is absolutely continuous with respect to the Lebesgue measure, then $\lim_{M\rightarrow \infty}d_{TV}\left(\mu_{B_M+T_M},\mu_{B+T}\right)=0$. Moreover under the same condition,  $\lim_{M\rightarrow \infty}d_{TV}\left(\mu_{c_MB_M},\mu_{c_0B}\right)=0$ if $\left(c_M\right)_{M\geq1} $ is a deterministic real sequence converging to $c_0\in\mathbb{R}^*$.
	\end{lemma}
	The first step in the proof of Proposition \ref{asymptoticRn} consists in applying Theorem  \ref{proho_generale} so that we can work with gaussian random vectors. Thus the new problem becomes to study the convergence in total variation of the law of a  given function of a normal random vector.
	The following lemma deals with this problem and its proof is given after the proof  of Proposition \ref{asymptoticRn}.
		
		\begin{lemma}\label{computationTV}
		Let 
			
			\begin{itemize}
				\item $(  \eta_{M,1},\cdots,\eta_{M,q+2}) _{M\geq 1} \subseteq \mathbb{R}^{q+2}$ such that for $i=1,\cdots,q+2$ $\lim\limits_{M\rightarrow \infty} \sqrt{M}\left( \frac{\eta_{M,i}}{M}-\eta_i\right) =0$ with  $(\eta_1,\cdots,\eta_{q+2}) \in \mathbb{R}^{q+2}$ 
				\item  $\left(W_1,\cdots, W_{q+2} \right)$  a zero-mean normal vector with a positive definite covariance matrix $\Sigma_1\in \mathbb{R}^{(q+2)\times(q+2)}$
				\item $\phi:\mathbb{R}\rightarrow\mathbb{R}$ be a measurable function differentiable at $\eta_{q+2}$ and such that $\phi(\eta_{q+2})\neq 0$.
			\end{itemize}

			If there exists
			$ \tilde{M}\in \mathbb{N}$ such that 
			
			\begin{equation}\label{dominateconvergence}	
			\int_{\mathbb{R}}\frac{1}{\left|\phi(\eta_{q+2}+y) \right|^{q+1}}e^{-\tilde{M}  y^2}dy < \infty ,
			\end{equation}
			the law of  the  following random vector 
			\begin{align}\label{vector1}
			\left(\frac{\phi\left( \frac{\eta_{M,q+2}}{M}+\frac{W_{q+2}}{\sqrt{M}}\right)\left(\sqrt{M}W_i+\eta_{M,i} \right) -M\phi (\eta_{q+2})\eta_i   }{\sqrt{M}} \right)_{1\leq i \leq q+1}  
			\end{align}
			
			converges in total variation as $M\rightarrow \infty$ to the law of $\left(\phi(\eta_{q+2})W_i + \eta_i\phi'(\eta_{q+2})W_{q+2}\right)_{1\leq i \leq q+1}$.	
		\end{lemma}
		\begin{remark}
			Let us observe that it is not difficult to  prove the pointwise convergence of (\ref{vector1}) to  
			$$\left(\phi(\eta_{q+2})W_i + \eta_i\phi'(\eta_{q+2})W_{q+2}\right)_{1\leq i \leq q+1}.$$
			Indeed for $i=1,\cdots,q+1$, using that  $\lim\limits_{M\rightarrow \infty} \sqrt{M}\left( \frac{\eta_{M,j}}{M}-\eta_j\right) =0$  and the hypothesis of differentiability of $\phi$ at $\eta_{q+2}$, one has

			\begin{align*}	&\frac{\phi\left( \frac{\eta_{M,q+2}}{M}+\frac{W_{q+2}}{\sqrt{M}}\right)\left(\sqrt{M}W_i+\eta_{M,i} \right) -M\phi (\eta_{q+2})\eta_i   }{\sqrt{M}}\\
			&\phantom{=}=\phi(\eta_{q+2})\left(W_i+\sqrt{M}\left(\frac{\eta_{M,i}}{M}-\eta_i \right) \right)+\phi'(\eta_{q+2})\left(W_{q+2}+\sqrt{M}\left( \frac{\eta_{M,q+2}}{M}-\eta_{q+2}\right)  \right)\left(\frac{W_i}{\sqrt{M}}+\frac{\eta_{M,i}}{M}\right) +o(1)\\
			&\phantom{=} \underset{M\rightarrow \infty}{\longrightarrow}\phi(\eta_{q+2})W_i + \eta_i\phi'(\eta_{q+2})W_{q+2}.
			\end{align*}
			The proof of the convergence in total variation will require more effort.
		\end{remark}
		\begin{remark}\label{defin.posit}
			Let us observe that the covariance matrix of the zero-mean normal random vector\\ $\left(\phi(\eta_{q+2})W_i + \eta_i\phi'(\eta_{q+2})W_{q+2}\right)_{1\leq i \leq q+1}$ is $A_1\Sigma_1A_1^T$ with   $A_1\in \mathbb{R}^{q+1\times(q+2)}$ the $q+1$-rank matrix given by 
			
			\begin{equation}\label{defA1}
			A_1=	\begin{pmatrix}
			\phi(\eta_{q+2}) & 0&0& 0&\cdots&0 &\eta_1\phi'(\eta_{q+2})\\
			0 & \phi(\eta_{q+2}) &0&\cdots&0&0&\eta_2\phi'(\eta_{q+2})\\
			0& 0& \phi(\eta_{q+2}) &0&0 &0& \cdots\\
			\cdots & \cdots& \cdots& \cdots&0 & \cdots & \cdots\\
			\cdots & \cdots& \cdots& \cdots&\cdots & \cdots & \cdots\\
			\cdots & \cdots& \cdots& \cdots&\phi(\eta_{q+2}) &0& \cdots\\
			0 & 0&0& 0&0&\phi(\eta_{q+2}) &\eta_{q+1}\phi'(\eta_{q+2})\\
			\end{pmatrix}.\end{equation}
			Therefore the covariance matrix is positive definite.
		\end{remark}
	The combination of Theorem \ref{proho_generale}, Lemma \ref{parthasarathy} and Lemma \ref{computationTV} allows to prove Proposition \ref{asymptoticRn}.
	
	\begin{proof}[Proof of Proposition \ref{asymptoticRn}] 
		Let
		\begin{align*}
		\bar\gamma &:=\left( \beta^1,\beta^2-\beta^1,\cdots, \beta^{q}-\beta^{q-1},\beta^{q+1}-\beta^{q}\right)\\
		&= \left( \beta^1-\beta^0,\beta^2-\beta^1,\cdots, \beta^{q}-\beta^{q-1},\beta^{q+1}-\beta^{q}\right)= \left( \gamma,\beta^{q+1}-\beta^{q}\right)
		\end{align*}
		with $\beta^0:=0$ by convention. 
		For $j=1,\cdots,q+1$ 
		let us consider the following independent $\mathbb{R}^{2}$-valued random vectors
		$$V_{M,j}:=\left(V_{M,j}^1,V_{M,j}^2\right) =\dfrac{1}{\sqrt{\bar\gamma_jM}}  \sum_{i=\beta_M^{j-1}+1}^{\beta_M^{j}}\left( Y_i-m_Y,Z_i-m_Z\right)  $$
		and let us rewrite  $\left(\frac{R_{M}- m_{Y} \phi(\beta^{q+1}m_{Y})M\gamma }{\sqrt{M}},K_M \right)$  in terms of $V_M:=(V_{M,j})_{1\leq j\leq q+1}$. One has 
		\begin{align*}
		\left(\frac{R_{M}- m_{Y} \phi(\beta^{q+1}m_{Y})M\gamma }{\sqrt{M}},K_M \right)= \left(\left( \frac{R_{\beta_M^{\ell-1:\ell}}- m_{Y} \phi(\beta^{q+1}m_{Y})\bar\gamma_\ell M }{\sqrt{M}}\right)_{1\leq \ell\leq q},K_M \right) =g_M\left( V_M\right) 
		\end{align*}
		with the component $g_{M,\ell}$ of $g_M:\mathbb{R}^{2(q+1)} \rightarrow \mathbb{R}^{q+1}$ for $\ell=1,\cdots,q$  given by:  $\forall c=\left( c_j^1,c_j^2\right)_{1\leq j \leq q+1} \in \mathbb{R}^{2(q+1)}$
		\begin{align} \label{g_M1} g_{M,\ell}(c)=\frac{\phi\left(  \frac{m_Y\beta^{q+1}_M}{M}+\frac{x}{M} + \frac{\sum\limits_{j=1}^{q+1}\sqrt{\bar\gamma_{j}}c_j^1}{\sqrt{M}}\right)\left(\sqrt{\bar\gamma_{\ell}M } c_{\ell}^1+m_Y\left(\beta^{\ell}_M-\beta^{\ell-1}_M \right) \right)-m_{Y}\phi(\beta^{q+1}m_{Y})\bar{\gamma}_\ell M}{\sqrt{M}},   
		\end{align}
		and  $\forall c=\left( c_j^1,c_j^2\right)_{1\leq j \leq q+1} \in \mathbb{R}^{2(q+1)}$
		
		\begin{align} \label{g_M2}
		g_{M,q+1}(c)=\frac{\phi\left(\frac{m_Y\beta^{q+1}_M}{M}+\frac{x}{M}+ \frac{\sum\limits_{j=1}^{q+1}\sqrt{\bar\gamma_{j}}c_j^1}{\sqrt{M}}\right) \left(\sqrt{M} \sum\limits_{j=1}^{q+1}\sqrt{\bar\gamma_{j}}c_j^2+m_Z\beta^{q+1}_M\right)-m_Z\phi\left( \beta^{q+1}m_Y\right)\beta^{q+1}M}{\sqrt{M}}.
		\end{align} 
		\\
		The purpose now is to ``asymptotically rewrite''  the vector $\left(\frac{R_{M}- m_{Y} \phi(\beta^{q+1}m_{Y})\gamma M }{\sqrt{M}},K_M \right)=g_M(V_M)$ in terms of a normal random vector so to apply Lemma \ref{computationTV}.
		By recalling that we denote by $\Sigma$ the covariance matrix of $(Y_i,Z_i)$, let $G=\left( G_1,\cdots, G_{q+1}\right)$ where the two dimensional vectors $G_{j}=\left(G_{j}^1,G_{j}^2 \right)$ $1\leq j\leq q+1$ are i.i.d. according to $ \mathcal{N}\left(0,\Sigma \right) $ .
		By the independence of the random vectors $V_{M,j}$, the independence of the random vectors $G_{j}$ and the well known fact that $d_{TV}\left(\prod_{i=1}^{\ell}\tilde{\nu}_i,\prod_{i=1}^{\ell}\nu_i \right)\leq \sum_{i=1}^{\ell}d_{TV}\left(\tilde{\nu}_i,\nu_i \right)$ with $\tilde{\nu}_i,\nu_i$ probability measures for $i=1,\cdots, \ell$, we have  
		
		\begin{equation} \label{convvector}
		d_{TV}(\mu_{V_M},\mu_{G})\leq \sum_{j=1}^{q+1} d_{TV}(\mu_{V_{M,j}},\mu_{G_{j}}).
		\end{equation}
		
		We are now going to prove that the right-hand side converges to $0$ as $M\rightarrow \infty$. 	Let us observe that for $j=1,\cdots,q+1$ 
		\begin{align} 
		V_{M,j}=& \frac{\sqrt{\beta_M^{j}-\beta_M^{j-1}}}{\sqrt{\bar\gamma_j M}}\times  \frac{1}{\sqrt{\beta_M^{j}-\beta_M^{j-1}}} \sum_{i=\beta_M^{j-1}+1}^{\beta_M^{j}}\left(  Y_i-m_Y,Z_i-m_Z\right) \label{Wn}
		\end{align}
		
		with $\lim_{M\rightarrow \infty}\dfrac{\sqrt{\beta_M^{j}-\beta_M^{j-1}}}{\sqrt{\bar\gamma_j M}} =1 $.
		
		Thanks to Theorem \ref{proho_generale} and the hypothesis that the law of $\left(Y_i,Z_i \right)$ has an absolutely continuous component, $\forall j=1,\cdots,q+1$  we have that the law of $\frac{1}{\sqrt{\beta_M^{j}-\beta_M^{j-1}}} \sum\limits_{i=\beta_M^{j-1}+1}^{\beta_M^{j}}\left( Y_i-m_Y,Z_i-m_Z\right)$ converges in total variation as $M\rightarrow \infty$ to the law of $G_{j}$
		and, by Lemma \ref{parthasarathy} and (\ref{Wn}), this implies 
		\begin{align*}
		\lim_{M\rightarrow \infty}d_{TV}\left(\mu_{V_{M,j}},\mu_{G_{j}}\right)=0, \quad j=1,\cdots,q+1.
		\end{align*}
		
		Hence we have proved the right-hand side of (\ref{convvector}) converges to $0$ as $M\rightarrow \infty$.
		Applying (\ref{pushforward!}) with $g=g_M$, we deduce that 
		\begin{align*}
		&d_{TV}\left( \mu_{g_{M}\left( V_M\right)},\mu_{g_{M}\left( G\right )} \right)\leq d_{TV}(\mu_{V_M},\mu_{G})\underset{M\rightarrow \infty}{\longrightarrow} 0  .
		\end{align*}
		
		We are now going to prove the convergence in total variation of the law of $g_{M}\left( G\right)$ and to do so we apply Lemma \ref{computationTV}.
		Recalling that  $G_{j}\sim \mathcal{N}\left(0,\Sigma \right) $ where $\Sigma$ is positive definite, $G_{j}$ i.i.d., it is possible to prove that the random variables $\sum\limits_{j=1}^{q+1}\sqrt{\bar\gamma_{j}}G_j^1,\sum\limits_{j=1}^{q+1}\sqrt{\bar\gamma_{j}}G_j^2,\sqrt{\bar{\gamma}_1}G_1^1,\cdots,\sqrt{\bar{\gamma}_q}G_q^1$ are linearly independent.

		Thus by using that $\lim\limits_{M\rightarrow \infty}\sqrt{M}\left( \frac{\beta_M^i}{M}- \beta^i\right) =0$ for $i=1,\cdots,q+1$ and  (\ref{dominateconvergence1}), we can apply Lemma \ref{computationTV} with 
		
		\begin{itemize}
			\item $(\eta_{M,1},\cdots,\eta_{M,q},\eta_{M,q+1}, \eta_{M,q+2})= \left(m_Y\left(\beta^{1}_M-\beta^{0}_M\right) ,\cdots,m_Y\left(\beta^{q}_M-\beta^{q-1}_M\right) ,m_Z \beta^{q+1}_M,\frac{m_Y\beta^{q+1}_M}{M}+\frac{x}{M}\right) $ \\ and  $(\eta_1,\cdots,\eta_{q},\eta_{q+1},\eta_{q+2}) = \left(m_Y\bar{\gamma}_1,\cdots,m_Y\bar{\gamma}_q,m_Z\beta^{q+1},m_Y\beta^{q+1} \right) $ 
			\item  $\left(W_1,\cdots,W_{q+2} \right)= \left(\sqrt{\bar{\gamma}_1}G_1^1,\cdots,\sqrt{\bar{\gamma}_q}G_q^1 ,\sum\limits_{j=1}^{q+1}\sqrt{\bar\gamma_{j}}G_j^2,\sum\limits_{j=1}^{q+1}\sqrt{\bar\gamma_{j}}G_j^1\right)$ 
		\end{itemize} 
		
		to conclude that the law of $g_M(G)$ converges in total variation as $M\rightarrow \infty$ to the law of $Y\sim \mathcal{N}\left(0,\Gamma \right)$ where $\Gamma\in \mathbb{R}^{(q+1)\times(q+1)}$ is positive definite by Remark \ref{defin.posit}. 
		
	\end{proof}
	The following lemma contains the key result to prove Lemma \ref{computationTV} . Its proof is provided in the Appendix.
	\begin{lemma}
		
		\label{decomposition}
		Let $(Y_M)_{M\geq 1}$ be a sequence of  $\mathbb{R}^d$-valued random vectors such that 
		
		\begin{equation} \label{conv.density}
		\liminf_{M\rightarrow \infty}p_{Y_M}(x)\geq p(x)\quad dx\,\,a.e. 
		\end{equation}
		with $p$ the density of an absolutely continuous $\mathbb{R}^d$-valued random variable $Y$. Then 
		
		\begin{enumerate}
			\item $\lim_{M\rightarrow \infty}\int_{\mathbb{R}^d}\left|  p_{Y_M}(x)-p(x)\right| dx =0$
			\item $\lim_{M\rightarrow \infty}d_{TV}\left(\mu_{Y_M},\mu_{Y} \right)=0.$ 
		\end{enumerate}
	\end{lemma}
	\begin{proof}[Proof of Lemma \ref{computationTV}]
		
		According to Lemma \ref{decomposition}, to prove the convergence in total variation of the law of the random vector (\ref{vector1}) it is enough to check that its density with respect to the Lebesgue measure converges pointwise to that of its limit.
		
		Let us preliminary observe that the density of the zero-mean normal vector $(W_{1},\cdots,W_{q+2})$ is given by 	 
		
		\begin{equation}\label{density}
		p(x_{1:q+2}) = \frac{1}{(2\pi)^{\frac{q+2}{2}}\times det(\Sigma_1)^{1/2}}e^{-\frac{1}{2}x_{1:q+2}^T\Sigma_1^{-1}x_{1:q+2}},\quad x_{1:q+2}\in \mathbb{R}^{q+2}
		\end{equation}
		where we denote the vector $(x_1,\cdots,x_{q+2})$ by $x_{1:q+2}$. 
		Let us moreover observe that by the property of positive definiteness of $\Sigma_1^{-1}$, the smallest eigenvalue $\lambda_1$ of $\Sigma_1^{-1}$ is positive and we have 
		
		\begin{equation}\label{defpos_prop}
		p(x_{1:q+2})  \leq \frac{1}{(2\pi)^{\frac{q+2}{2}}\times det(\Sigma_1)^{1/2}} e^{-\frac{1}{2}\lambda_1\left( \sum\limits_{i=1}^{q+2}x_i^2\right) }, \quad x_{1:q+2}\in \mathbb{R}^{q+2}.
		\end{equation}
		
		Let us now compute the density of (\ref{vector1}). Let $f:\mathbb{R}^{q+1} \rightarrow \mathbb{R}_+$ be a non-negative
		measurable function. Writing the expectation

		\begin{equation}\label{expectation1}
		\mathbb{E}\left( f	\left(\left(\frac{\phi\left( \frac{\eta_{M,q+2}}{M}+\frac{W_{q+2}}{\sqrt{M}}\right)\left(\sqrt{M}W_i+\eta_{M,i} \right) -M\phi (\eta_{q+2})\eta_i   }{\sqrt{M}} \right)_{1\leq i \leq q+1} \right)  \right)
		\end{equation}
		as an integral with respect to the density $p$ and then applying  the change of variable 
		
		\begin{align*}
		\xi_i=\frac{\phi\left( \frac{\eta_{M,q+2}}{M}+\frac{x_{q+2}}{\sqrt{M}}\right)\left(\sqrt{M}x_i+\eta_{M,i} \right) -M\phi (\eta_{q+2})\eta_{i}   }{\sqrt{M}}\quad i=1,\cdots,q+1
		\end{align*}
		
		with inverse
		
		\begin{equation}
		x_{M,i}:= \frac{1}{\sqrt{M}}\left(\frac{\sqrt{M}\xi_i+M\phi (\eta_{q+2})\eta_{i}}{\phi\left( \frac{\eta_{M,q+2}}{M}+\frac{x_{q+2}}{\sqrt{M}}\right)}-\eta_{M,i}\right) \quad i=1,\cdots,q+1
		\end{equation}
		for $x_{{q+2}}\in \mathbb{R}$ outside the set $\left\lbrace t\in \mathbb{R} : \phi\left( \frac{\eta_{M,q+2}}{M}+\frac{t}{\sqrt{M}}\right) =0 \right\rbrace $ which is Lebesgue negligible since by a change of variable and (\ref{dominateconvergence})
		\begin{align*}
		\int_{\mathbb{R}}\frac{1}{\left| \phi\left( \frac{\eta_{M,q+2}}{M}+\frac{t}{\sqrt{M}}\right)\right|^q+1}e^{-\tilde{M}\left( \frac{\eta_{M,q+2}}{M}-\eta_{q+2}+\frac{t}{\sqrt{M}}\right) ^2}dt	=\sqrt{M}\int_{\mathbb{R}}\frac{1}{\left|\phi(\eta_{q+2}+y) \right|^{q+1}}e^{-\tilde{M}  y^2}dy < \infty,
		\end{align*}
		we get
		\begin{align*}
		&\int_{\mathbb{R}}dx_{q+2}\int_{\mathbb{R}^{q+1}}\frac{1}{\left|\phi\left( \frac{\eta_{M,q+2}}{M}+\frac{x_{q+2}}{\sqrt{M}}\right)\right| ^{q+1} } f(\xi_{1:q+1}) p\left(x_{M,1},\cdots,x_{M,q+1},x_{q+2}\right)d\xi_{1:q+1}.
		\end{align*}
		We have therefore obtained that the density of $\left(\frac{\phi\left( \frac{\eta_{M,q+2}}{M}+\frac{W_{q+2}}{\sqrt{M}}\right)\left(\sqrt{M}W_i+\eta_{M,i} \right) -M\phi (\eta_{q+2})\eta_i   }{\sqrt{M}} \right)_{1\leq i \leq q+1}$  at  $\xi \in \mathbb{R}^{q+1}$ is given by
		\begin{align}
		&\int_{\mathbb{R}}\frac{1}{\left|\phi\left( \frac{\eta_{M,q+2}}{M}+\frac{x_{q+2}}{\sqrt{M}}\right)\right| ^{q+1} }  p\left(x_{M,1},\cdots,x_{M,q+1},x_{q+2}\right)dx_{q+2}. \label{densityvector}
		\end{align} 
		Since, by hypothesis, $\phi$ is continuous at $\eta_{q+2}$ with  $\phi(\eta_{q+2})\neq 0$, there exists $\delta > 0$ such that $\forall t\in \mathbb{R}\, : \left|t\right|\leq \delta$
		\begin{equation}\label{continuitypsi}
		\left|\phi(\eta_{q+2}+t) \right|\geq \frac{\left|\phi(\eta_{q+2}) \right|}{2}.
		\end{equation}
		
		Let us now study the pointwise convergence of (\ref{densityvector}). Let $M\geq 1$ and let us rewrite the integral as  
		\begin{align}
		\int_{\mathbb{R}}\frac{1}{\left|\phi\left( \frac{\eta_{M,q+2}}{M}+\frac{x_{q+2}}{\sqrt{M}}\right)\right| ^{q+1} }  p\left(x_{M,1:q+1},x_{q+2}\right)dx_{q+2}&= \int_{\left|x_{q+2}\right|\leq \frac{\delta}{2}\sqrt{M}} \frac{1}{\left|\phi\left( \frac{\eta_{M,q+2}}{M}+\frac{x_{q+2}}{\sqrt{M}}\right)\right| ^{q+1} }  p\left(x_{M,1:q+1},x_{q+2}\right)dx_{q+2}\label{doubleint1}\\
		&\phantom{==}+ \int_{\left|x_{q+2}\right|> \frac{\delta}{2}\sqrt{M}}\frac{1}{\left|\phi\left( \frac{\eta_{M,q+2}}{M}+\frac{x_{q+2}}{\sqrt{M}}\right)\right| ^{q+1} }  p\left(x_{M,1:q+1},x_{q+2}\right)dx_{q+2}. \label{doubleint1.1}
		\end{align}\\
		
		Let us start studying the convergence as $M\rightarrow \infty$ of (\ref{doubleint1}). Thanks to  (\ref{continuitypsi})  and (\ref{defpos_prop}) and for $M$ big enough so that $ \left|\frac{\eta_{M,q+2}}{M}-\eta_{q+2}\right|  \leq \frac{\delta}{2}$, we have
		\begin{align*}
		&1_{\left\lbrace \left|x_{q+2}\right|\leq \frac{\delta}{2} \sqrt{M}\right\rbrace }\frac{1}{\left|\phi\left( \frac{\eta_{M,q+2}}{M}+\frac{x_{q+2}}{\sqrt{M}}\right)\right| ^{q+1} }  p\left(x_{M,1:q+1},x_{q+2}\right)\leq\frac{2^{q+1}}{\left( 2\pi\right) ^{\frac{q+2}{2}}\times det(\Sigma_1)^{\frac{1}{2}}\left|\phi(\eta_{q+2}) \right|^{q+1}}e^{-\frac{1}{2}\lambda_1 x_{q+2}^2}
		\end{align*}
		and since by hypothesis  $\phi$ is continuous at $\eta_{q+2}$, for each $x_{q+2}\in \mathbb{R}$ we have
		
		\begin{equation} 
		\phi\left( \frac{\eta_{M,q+2}}{M}+ \frac{x_{q+2}}{\sqrt{M}}\right)  \underset{M\rightarrow \infty}{\longrightarrow} \phi(\eta_{q+2}).
		\end{equation}
		
		Moreover for each $x_{q+2}\in \mathbb{R}$, using that  $\lim\limits_{M\rightarrow \infty} \sqrt{M}\left( \frac{\eta_{M,i}}{M}-\eta_i\right) =0$  for $i=1,\cdots,q+2$ and the hypothesis of differentiability of $\phi$ one has
		\begin{align} \nonumber
		x_{M,i} &= \frac{1}{\sqrt{M}}\left(\frac{\sqrt{M}\xi_i+M\phi (\eta_{q+2})\eta_{i}}{\phi\left( \frac{\eta_{M,q+2}}{M}+\frac{x_{q+2}}{\sqrt{M}}\right)}-\eta_{M,i}\right)\\\nonumber
		&=\left(\frac{\xi_i+\phi\left(\eta_{q+2}\right) \sqrt{M}\left(\eta_i-\frac{\eta_{M,i} }{M} \right) -\phi'\left(\eta_{q+2}\right)\frac{\eta_{M,i}}{\sqrt{M}}\left(\frac{x_{q+2}}{\sqrt{M}}+\frac{\eta_{M,q+2}}{M} -\eta_{q+2}\right)  +o(1)}{\phi(\eta_{q+2})+\phi'(\eta_{q+2})\left(\frac{x_{q+2}}{\sqrt{M}}+\frac{\eta_{M,q+2}}{M}-\eta_{q+2} \right) +o\left( \frac{1}{\sqrt{M}}\right)} \right)\\ 
		&\phantom{=}\underset{M\rightarrow \infty }{\longrightarrow}\frac{\xi_i-\phi'\left(\eta_{q+2}\right)\eta_{i}x_{q+2}}{\phi(\eta_{q+2})} . \label{limite1}
		\end{align}

		We can therefore apply Lebesgue's theorem to (\ref{doubleint1}) and obtain that for each $\xi\in\mathbb{R}^{q+1}$\\
		\begin{align*}
		&\lim_{M\rightarrow \infty} \int_{\left|x_{q+2}\right|\leq \frac{\delta}{2}\sqrt{M}} \frac{1}{\left|\phi\left( \frac{\eta_{M,q+2}}{M}+\frac{x_{q+2}}{\sqrt{M}}\right)\right| ^{q+1} }  p\left(x_{M,1:q+1},x_{q+2}\right)dx_{q+2}\\
		& = \frac{1}{ \left| \phi\left( \eta_{q+2}\right) \right| ^{q+1} }\int_{\mathbb{R}} p\left(\left(\frac{\xi_i-\phi'\left(\eta_{q+2}\right)\eta_{i}x_{q+2}}{\phi(\eta_{q+2})} \right)_{1\leq i \leq q+1},x_{q+2}\right)dx_{q+2}\\
		&=\int_{\mathbb{R}} p_{\left(\left(\phi(\eta_{q+2})W_{i}+\phi'(\eta_{q+2})\eta_iW_{q+2} \right)_{1\leq i \leq q+1},W_{q+2} \right)} (\xi_{1:q+1},x_{q+2})dx_{q+2} \\
		&= p_{\left(\left(\phi(\eta_{q+2})W_{i}+\phi'(\eta_{q+2})\eta_iW_{q+2} \right)_{1\leq i \leq q+1} \right)} (\xi_{1:q+1}).
		\end{align*}
		
		Let us now prove that (\ref{doubleint1.1}) converges to $0$ as $M\rightarrow \infty$.  By applying a change of variable and by choosing $M$ big enough so that $\left|\frac{\eta_{M,q+2}}{M}-\eta_{q+2} \right|\leq \frac{\delta}{4} $, we obtain 
		\begin{align*}
		\int_{\left|x_{q+2}\right|> \frac{\delta}{2}\sqrt{M}}\frac{1}{\left|\phi\left( \frac{\eta_{M,q+2}}{M}+\frac{x_{q+2}}{\sqrt{M}}\right)\right| ^{q+1} }&  p\left(x_{M,1:q+1},x_{q+2}\right)dx_{q+2}= \int_{\left|y\right|> \frac{\delta}{2}} \frac{\sqrt{M}}{\left| \phi\left(  \frac{\eta_{M,q+2}}{M}+y\right)\right| ^{q+1}} p\left(x_{M,1:q+1},\sqrt{M} y\right) dy\\
		&\leq \int_{\left|z\right|> \frac{\delta}{4}} \frac{\sqrt{M}}{\left| \phi\left( \eta_{q+2}+ z\right)\right| ^{q+1}} p\left( x_{M,1:q+1},\sqrt{M} z+\sqrt{M}\left(\eta_{q+2}-\frac{\eta_{M,q+2}}{M} \right)\right)dz.
		\end{align*}
		Therefore by (\ref{defpos_prop}) and by using that $\forall x_1,x_2\in \mathbb{R}$ $(x_1-x_2)^2 \geq \frac{x_1^2}{2}-x_2^2$, one has 
		\begin{align*}
		&1_{\left\lbrace|z|>\frac{\delta}{4} \right\rbrace }  \frac{\sqrt{M}}{\left| \phi\left( \eta_{q+2}+ z\right)\right| ^{q+1}} p\left( x_{M,1:q+1},\sqrt{M} z+\sqrt{M}\left(\eta_{q+2}-\frac{\eta_{M,q+2}}{M} \right)\right) \\
		&\leq 1_{\left\lbrace|z|>\frac{\delta}{4} \right\rbrace }\frac{\sqrt{M}}{\left( 2\pi\right) ^\frac{q+2}{2} \times det(\Sigma_1)^\frac{1}{2}\left| \phi\left(\eta_{q+2}+ z\right)\right| ^{q+1}}e^{-\frac{1}{2}\lambda_1 \left(\sqrt{M} z-\sqrt{M}\left(\frac{\eta_{M,q+2}}{M} -\eta_{q+2}\right)  \right)^2 }\\
		&\leq 1_{\left\lbrace|z|>\frac{\delta}{4} \right\rbrace }\frac{\sqrt{M}}{\left( 2\pi\right) ^\frac{q+2}{2} \times det(\Sigma_1)^\frac{1}{2}\left| \phi\left( \eta_{q+2}+ z\right)\right| ^{q+1}}e^{-\frac{1}{4}\lambda_1 Mz^2 } e^{\frac{1}{2}\lambda_1 \left(\sqrt{M}\left(\frac{\eta_{M,q+2}}{M}-\eta_{q+2} \right)  \right)^2 } \\
		&\leq C1_{\left\lbrace|z|>\frac{\delta}{4} \right\rbrace }\frac{\sqrt{M}}{\left( 2\pi\right) ^\frac{q+2}{2} \times det(\Sigma_1)^\frac{1}{2}\left| \phi\left( \eta_{q+2}+ z\right)\right| ^{q+1}}e^{-\frac{1}{4}\lambda_1 Mz^2 } 
		\end{align*}
		where for the last inequality we use that the sequence  $\left( \sqrt{M}\left(\frac{\eta_{M,q+2}}{M}-\eta_{q+2} \right)\right)_{M\geq 1}$  converges to $0$ and so in particular $e^{\frac{1}{2}\lambda_1 \left(\sqrt{M}\left(\frac{\eta_{M,q+2}}{M}-\eta_{q+2} \right)  \right)^2 }$ is bounded by a positive constant $C< \infty$. The right-hand side converges to $0$ as $M\rightarrow\infty$.\\
		Let us now observe that for $|z|>\frac{\delta}{4}$ and $t\geq \frac{32}{\delta^2\lambda_1}$,  $$\partial_t(\sqrt{t}e^{-\frac{1}{4}\lambda_1 tz^2})=\frac{e^{-\frac{1}{4}\lambda_1 z^2t}}{4} \bigg(\frac{2}{\sqrt{t}}-z^2\sqrt{t}\lambda_1\bigg)\leq 0.$$
		Therefore if $M\geq \bar{M}:=\max \left( \left\lceil\frac{4\tilde{M}}{\lambda_1}\right\rceil,\left\lceil\frac{32}{\delta^2\lambda_1}\right\rceil\right) $, where $\tilde{M}$  is defined in the statement of the Lemma \ref{computationTV},
		
		\begin{align}\label{magg}
		1_{\left\lbrace|z|>\frac{\delta}{4} \right\rbrace }\frac{\sqrt{M}}{\left| \phi\left(\eta_{q+2}+ z\right)\right| ^{q+1}}e^{-\frac{1}{4}\lambda_1 Mz^2 }  &\leq 1_{\left\lbrace|z|>\frac{\delta}{4} \right\rbrace }\frac{\sqrt{\bar{M}}}{\left| \phi\left( \eta_{q+2}+ z\right)\right| ^{q+1}}e^{-\frac{1}{4}\lambda_1 \bar{M}z^2}\\
		&\leq  1_{\left\lbrace|z|> \frac{\delta}{4} \right\rbrace }\frac{\sqrt{\bar{M}}}{\left| \phi\left( \eta_{q+2}+ z\right)\right| ^{q+1}}e^{-\tilde{M} z^2} \label{domin}
		\end{align}  
		where the last term is integrable by hypothesis (\ref{dominateconvergence}).   
		By Lebesgue's theorem  we can therefore conclude that $\int_{\left|z\right|> \frac{\delta}{4}} \frac{\sqrt{M}}{\left| \phi\left( \eta_{q+2}+ z\right)\right| ^{q+1}} p\left( x_{M,1:q+1},\sqrt{M} z+\sqrt{M}\left(\eta_{q+2}-\frac{\eta_{M,q+2}}{M} \right)\right)dz$ tends to $0$ as $M\rightarrow \infty$.\\ In conclusion, we have obtained that  the density of $ \left(\frac{\phi\left( \frac{\eta_{M,q+2}}{M}+\frac{W_{q+2}}{\sqrt{M}}\right)\left(\sqrt{M}W_i+\eta_{M,i} \right) -M\phi (\eta_{q+2})\eta_i   }{\sqrt{M}} \right)_{1\leq i \leq q+1} $  converges pointwise as $M\rightarrow \infty$ to 
		the density of $\left(\left(\phi(\eta_{q+2})W_{i}+\phi'(\eta_{q+2})\eta_iW_{q+2} \right)_{1\leq i \leq q+1} \right)$. 
	\end{proof}
We are now going to prove Theorem \ref{weyl.criterion}.

\begin{proof}[Proof of Theorem \ref{weyl.criterion}]
	We denote $\left(\left\lbrace B_{M}\right\rbrace ,H_M \right):= \left(\left\lbrace B^1_{M}\right\rbrace ,\cdots, \left\lbrace B^q_{M}\right\rbrace, H_M \right)$. Let us first observe that for every $k\in \mathbb{Z}^q, u \in \mathbb{R}^{q'}$  one has
	\begin{align}\label{proprieta}
	\phi_{\left(\left\lbrace B_{M}\right\rbrace ,H_M \right)}\left(2\pi k,u \right) = \mathbb{E}\left(e^{i2\pi k^T\left\lbrace B_{M}\right\rbrace} e^{iu^TH_M}\right)  =  \mathbb{E}\left(e^{i2\pi k^T B_{M}} e^{-i2\pi\sum\limits_{j=1}^q k_j \floor*{B_{M}^j}} e^{iu^TH_M}\right) = \phi_{\left( B_{M} ,H_M \right)}\left(2\pi k,u \right).
	\end{align}
	
	If $\left(\left\lbrace B_{M}\right\rbrace ,H_M \right)$ converges in distribution as $M\rightarrow \infty$ to $\left(U,H \right) $ where $U$ is independent of $H$ and uniformly distributed on $\left[0,1 \right]^q $, $\forall (k,u):=(k_1,\cdots,k_q,u)\in \left\lbrace \mathbb{Z}^q\setminus \left\lbrace 0\right\rbrace \right\rbrace  \times \mathbb{R}^{q'}$ one has 
	\begin{align} \label{Znumber}
	\lim_{M\rightarrow \infty} \phi_{\left( B_{M} ,H_M \right)}\left(2\pi k,u \right) = \lim_{M\rightarrow \infty}\phi_{\left(\left\lbrace B_{M}\right\rbrace ,H_M \right)}\left(2\pi k,u \right)= \phi_{\left(U ,H \right)}\left(2k\pi,u \right) = \mathbb{E}\left( e^{i2\pi k U}\right) \mathbb{E}\left(e^{iuH} \right)=0.
	\end{align}
	
	Conversely, assume that (\ref{suffcond}) holds for every $(k,u)\in \left\lbrace \mathbb{Z}^q\setminus \left\lbrace 0\right\rbrace \right\rbrace  \times \mathbb{R}^{q'}$ and let us prove that $\forall (u_1,u_2)\in  \mathbb{R}^q  \times \mathbb{R}^{q'}$, 
	\begin{align*}
	\lim_{M\rightarrow \infty}\phi_{\left(\left\lbrace B_{M}\right\rbrace ,H_M\right)}(u_1,u_2)=\phi_{U}(u_1)\phi_{H}(u_2)
	\end{align*}
	with $(U,H)$ as above.\\
	Fix $(u_1,u_2)\in \left\lbrace \mathbb{R}^q\setminus \left\lbrace 0\right\rbrace\right\rbrace  \times \mathbb{R}^{q'}$. Given $0 <\epsilon < \frac{1}{4}$, let us define for $j=1,\cdots,q$ a complex $\lceil\frac{q+1}{2} \rceil$-times continuously differentiable function $\phi_{\epsilon,j}$ defined on $\left[ 0,1\right]$  that coincides with $x_j\mapsto e^{iu_1^jx_j}$ on  $\left[0,1-\epsilon \right]$ and such that $\phi_{\epsilon,j}(1)=\phi_{\epsilon,j}(0)$.
	It is possible to choose $\phi_{\epsilon,j}$ such that $\sup_{x\in \left[0,1 \right] }\left| \phi_{\epsilon,j}(x)\right| \leq C^{\frac{1}{q}}$, with $C< \infty$ not depending on $\epsilon$ and $j$. If we now consider the function  $\phi_{\epsilon}\left(x_1,\cdots,x_q \right):=\phi_{\epsilon,1}\left(x_1 \right)\cdot\cdot\cdot\phi_{\epsilon,q}\left(x_q \right) $ $\forall \left(x_1,\cdots,x_q \right) \in \left[ 0,1\right]^q$, then it is  $\lceil\frac{q+1}{2} \rceil$-times continuously differentiable, coincides with $\left(x_1,\cdots,x_q \right)\mapsto e^{iu_1^1x_1}\cdots e^{iu_1^qx_q}$ on $\left[0,1-\epsilon \right]^q$ and   $\phi_{\epsilon}(x_1,\cdots,x_{j-1},0,x_{j+1},\cdots,x_q)=\phi_{\epsilon}(x_1,\cdots,x_{j-1},1,x_{j+1},\cdots,x_q)$ $\forall j,$ $\forall \left(x_1,\cdots,x_{j-1},x_{j+1},\cdots,x_q \right) \in \left[ 0,1\right]^{q-1}$.   By the theory of Fourier's series (see for instance the Section Sobolev Spaces in Chapter 5 of \cite{Roe}), there exists $M^\epsilon\in \mathbb{N}$  such that 
	
	\begin{equation} \label{fourier}
	\sup_{x\in \left[0,1 \right]^q }\left|\phi_\epsilon(x) - \sum_{\max\limits_{j=1,\cdots,q}\left|k_j \right| \leq M^\epsilon}^{}c^\epsilon_ke^{i2\pi k^T x}\right|=	\sup_{x\in \left[0,1 \right]^q }\left|\phi_\epsilon(x) - \sum_{\left|k_1 \right| \leq M^\epsilon}^{}\cdots  \sum_{\left|k_q \right| \leq M^\epsilon}^{}c^\epsilon_ke^{i2\pi k^T x}\right| \leq \epsilon
	\end{equation}
	
	where for $k\in \mathbb{Z}^q$
	
	$$c^\epsilon_k = \int_{\left[ 0,1\right]^q}\phi_\epsilon(x)e^{-i2\pi k^T x}dx $$ 
	satisfies $\left| c^\epsilon_k\right| \leq C$.
	Let $\nu_M$ denote the law of $\left(\left\lbrace B_{M}\right\rbrace ,H_M\right) $ (having support in $\left[ 0,1\right]^q \times \mathbb{R}^{q'} $).\\ Given $M\geq 1$  one has
	\begin{align*}
	&\left| \phi_{\left(\left\lbrace B_{M}\right\rbrace ,H_M\right)}(u_1,u_2)- \phi_{U}(u_1)\phi_{H}(u_2) \right|  \\
	&\leq \left| \int_{\left[ 0,1\right]^q \times \mathbb{R}^{q'}} \left( e^{iu_1^Tx}- \phi_{\epsilon}(x)\right)e^{iu_2^Ty} \nu_M(dx,dy) \right|\\
	&\phantom{=}+\left| \int_{\left[ 0,1\right]^q \times \mathbb{R}^{q'}}  \phi_{\epsilon}(x)e^{iu_2^Ty} \nu_M(dx,dy) - \sum_{\max\limits_{j=1,\cdots,q}\left|k_j \right| \leq M^\epsilon}^{}c^\epsilon_k \int_{\left[ 0,1\right]^q \times \mathbb{R}^{q'}}  e^{i2\pi k^T x}e^{iu_2^Ty} \nu_M(dx,dy) \right| \\
	& \phantom{=}+ \left| \sum_{\max\limits_{j=1,\cdots,q}\left|k_j \right| \leq M^\epsilon,k\neq 0}^{}c^\epsilon_k \int_{\left[ 0,1\right]^q \times \mathbb{R}^{q'}}  e^{i2\pi k^Tx}e^{iu_2^Ty} \nu_M(dx,dy)\right|+\left|  c^\epsilon_0\phi_{H_M}(u_2)-c^\epsilon_0\phi_{H}(u_2)\right|\\
	& \phantom{=} +\left|  c^\epsilon_0\phi_{H}(u_2)-\phi_{U}(u_1)\phi_{H}(u_2) \right| \\
	&\leq (1+C)\int_{\left[ 0,1\right]^q\setminus\left[ 0,1-\epsilon\right]^q}\nu_M(dx,\mathbb{R}^{q'})+\epsilon+\left| \sum_{\max\limits_{j=1,\cdots,q}\left|k_j \right| \leq M^\epsilon,k\neq 0}^{}c^\epsilon_k \phi_{(B_M,H_M)}(2 k\pi,u_2) \right|\\
	&\phantom{=}+ C \left| \phi_{H_M}(u_2)-\phi_{H}(u_2)\right| + \left| \phi_{U}(u_1) -c^\epsilon_0 \right|.  
	\end{align*}

	Let us look in more detail at the first term of the right-hand side. Let $\bar{\phi}_\epsilon$ be a real continuously differentiable function   defined on $[0,1]$  that is equal to $0$ on $\left[2\epsilon,1-2\epsilon \right] $  and equal to $1$ on $\left[0,\epsilon \right] $ and  $\left[1-\epsilon,1 \right] $. By the theory of Fourier's series, there exists $\bar{M}^\epsilon\in \mathbb{N}$  such that
	\begin{equation} \label{fourier2}
	\sup_{x\in \left[0,1 \right] }\left|\bar{\phi}_\epsilon(x) - \sum_{\left|\ell \right| \leq \bar{M}^\epsilon}^{}d^\epsilon_\ell e^{i2\pi \ell x}\right|\leq \epsilon
	\end{equation}
	
	where for $\ell\in \mathbb{Z}$
	
	$$d^\epsilon_\ell = \int_{\left[ 0,1\right]}\bar\phi_\epsilon(x)e^{-i2\pi \ell x}dx. $$ 
	One has 
	\begin{align*}
	\int_{\left[ 0,1\right]^q\setminus\left[ 0,1-\epsilon\right]^q} \nu_M(dx,\mathbb{R}^{q'})  &\leq  \sum\limits_{j=1}^{q}\int_{\left[ 1-\epsilon,1\right]} \mu_{\left\lbrace B_M^j\right\rbrace}(dx)\leq  \sum\limits_{j=1}^{q}\int_{\left[ 0,1\right]}\bar{\phi}_\epsilon\left( x\right) \mu_{\left\lbrace B_M^j\right\rbrace}(dx)\\
	& \leq \sum\limits_{j=1}^{q}\int_{\left[ 0,1\right]} \left| \bar{\phi}_\epsilon\left( x\right)- \sum_{\left|\ell \right| \leq \bar{M}^\epsilon}^{}d^\epsilon_\ell e^{i2\pi \ell x}\right|  \mu_{\left\lbrace B_M^j\right\rbrace}(dx)+\sum\limits_{j=1}^{q}\left| \sum_{\left|\ell \right| \leq \bar{M}^\epsilon,\ell\neq 0}^{} d^\epsilon_\ell \int_{\left[ 0,1\right]}e^{i2\pi \ell x}\mu_{\left\lbrace B_M^j\right\rbrace}(dx)\right| \\
	&\phantom{==}+q\left| d^\epsilon_0\right|\\
	&\leq q\epsilon+\sum\limits_{j=1}^{q}\left| \sum_{\left|\ell \right| \leq \bar{M}^\epsilon,\ell\neq 0}^{} d^\epsilon_\ell\phi_{B_M^j}\left( 2\pi\ell\right)\right| +q\left| d^\epsilon_0\right|.
	\end{align*}
	Thus we have obtained that
	
	\begin{align*}
	&\left| \phi_{\left(\left\lbrace B_{M}\right\rbrace ,H_M\right)}(u_1,u_2)- \phi_{U}(u_1)\phi_{H}(u_2) \right|  \\
	&\leq (1+C)\left(q\epsilon+\sum\limits_{j=1}^{q}\left| \sum_{\left|\ell \right| \leq \bar{M}^\epsilon,\ell\neq 0}^{} d^\epsilon_\ell\phi_{B_M^j}\left( 2\pi\ell\right)\right| +q\left|d^\epsilon_0\right| \right) +\epsilon+\left| \sum_{\max\limits_{j=1,\cdots,q}\left|k_j \right| \leq M^\epsilon,k\neq 0}^{}c^\epsilon_k \phi_{(B_M,H_M)}(2 \pi k,u_2) \right|\\
	&\phantom{=}+ C \left| \phi_{H_M}(u_2)-\phi_{H}(u_2)\right| + \left| \phi_{U}(u_1) -c^\epsilon_0 \right|.  
	\end{align*}
	
	Given $0<\epsilon<\frac{1}{4}$, we can first take the superior limit as $M\rightarrow \infty$ of the right-hand side.\\ 
	By (\ref{suffcond}) we have 
	\begin{align*}
	\lim_{M\rightarrow \infty}\sum\limits_{j=1}^{q}\left| \sum_{\left|\ell \right| \leq \bar{M}^\epsilon,\ell\neq 0}^{} d^\epsilon_\ell\phi_{B_M^j}\left( 2\pi\ell\right)\right| = 0 = \lim_{M\rightarrow \infty}\left| \sum_{\max\limits_{j=1,\cdots,q}\left|k_j \right| \leq M^\epsilon,k\neq 0}^{}c^\epsilon_k \phi_{(B_M,H_M)}(2\pi k,u_2) \right|
	\end{align*}
	and by the hypothesis that $H_M\overset{d}{\Rightarrow}H$ we have $\lim\limits_{M\rightarrow \infty}\left| \phi_{H_M}(u_2)-\phi_{H}(u_2)\right|=0$. We therefore obtain 
	\begin{align} \label{afterM}
	\limsup_{M\rightarrow \infty}\left| \phi_{\left(\left\lbrace B_{M}\right\rbrace ,H_M\right)}(u_1,u_2)- \phi_{U}(u_1)\phi_{H}(u_2) \right| \leq q(1+C)\left(  \epsilon+\left| d^\epsilon_0\right| \right)  + \epsilon +\left| \phi_{U}(u_1) -c^\epsilon_0 \right|.
	\end{align}
	
	Now since
	\begin{align*}
	\left|d^\epsilon_0  \right| 	= 	d^\epsilon_0 =  \int_{\left[ 0,1\right]}\bar\phi_\epsilon(x)dx \leq 4\epsilon 
	\end{align*}
	and 
	\begin{align*}
	\left| \phi_{U}(u_1) -c^\epsilon_0 \right|&=\left|\int_{\left[ 0,1\right]^q}(e^{iu_1^Tx}-\phi_\epsilon(x))dx \right|\leq \left(C+1 \right) \int_{\left[ 0,1\right]^q\setminus\left[ 0,1-\epsilon\right]^q}dx
	\end{align*} 
	with  $\int_{\left[ 0,1\right]^q\setminus\left[ 0,1-\epsilon\right]^q} dx= 1-\left(1-\epsilon \right)^q  \underset{\epsilon\rightarrow 0}{\longrightarrow} 0$, the limit as $\epsilon \rightarrow 0$ of the right-hand side of (\ref{afterM}) is $0$.\\
	If $u_1=0$ and $u_2\in \mathbb{R}$, by the hypothesis that $H_M\overset{d}{\Rightarrow}H$, we directly obtain 
	\begin{align*}
	\phi_{\left(\left\lbrace B_{M}\right\rbrace ,H_M\right)}(0,u_2) = \phi_{H_M}(u_2)\underset{M\rightarrow \infty}{\longrightarrow} \phi_{H}(u_2) = \phi_{U}(0)\phi_{H}(u_2).
	\end{align*}.
\end{proof}

\section{Appendix}
\begin{lemma}\label{tvdistancecomplex}
	Given the measures $\mu_1$ and $\mu_2$ on $\mathbb{R}^d$, one has
	\begin{equation}\label{tvcomplex}
	\sup_{\left\|f\right\|_{\infty,\mathbb{C}}\leq 1}\left|\int_{\mathbb{R}^d}f(x)\left( \mu_1(dx)-\mu_2(dx) \right) \right|= 2d_{TV}\left(\mu_1,\mu_2 \right)
	\end{equation}
	where the supremum is taken over the complex-valued measurable functions.
\end{lemma}

\begin{proof}[Proof of Lemma \ref{tvdistancecomplex}] By observing that the set of complex-valued measurable functions contains the set of real-valued measurable functions and that $\left\|f\right\|_{\infty,\mathbb{C}} =\left\|f\right\|_{\infty,\mathbb{R}}$  when $f$ is real-valued, thanks to (\ref{dtvr}) we can easily obtain that $$\sup_{\left\|f\right\|_{\infty,\mathbb{C}}\leq 1}\left|\int_{\mathbb{R}^d}f(x)\left( \mu_1(dx)-\mu_2(dx) \right) \right|\geq 2d_{TV}\left(\mu_1,\mu_2 \right). $$ 
	For what concerns the other inequality, let $f$ be a complex-valued measurable function such that $\left\|f\right\|_{\infty,\mathbb{C}}\leq1$. Then there exist $\psi\geq 0$ and $\rho \in [0,2\pi)$ such that 
	
	\begin{align*}
	\int_{\mathbb{R}^d}f(x)\left( \mu_1(dx)-\mu_2(dx) \right) = \psi e^{i\rho}.
	\end{align*} 
	Note that $\left\|\mathfrak{Re}(e^{-i\rho}f)\right\|_{\infty,\mathbb{R}}\leq \left\|e^{-i\rho}f\right\|_{\infty,\mathbb{C}}=\left\|f\right\|_{\infty,\mathbb{C}}\leq 1$. Hence
	
	\begin{align*}
	\left| \int_{\mathbb{R}^d}f(x)\left( \mu_1(dx)-\mu_2(dx) \right)\right| & = 	\int_{\mathbb{R}^d} \mathfrak{Re}(e^{-i\rho}f(x))\left( \mu_1(dx)-\mu_2(dx) \right)\\
	&\leq	\sup_{\left\|g\right\|_{\infty,\mathbb{R}}\leq 1}\left|\int_{\mathbb{R}^d}g(x)\left( \mu_1(dx)-\mu_2(dx) \right) \right|= 2d_{TV}\left(\mu_1,\mu_2 \right) 
	\end{align*}
	
	where for the last equality we use (\ref{dtvr}). Since $f$ is arbitrary, we can conclude the proof.
\end{proof}

\begin{proof}[Proof of Lemma \ref{abs_continuity}]
	Let $\xi$ be an absolutely continuous random variable independent of $\left( Y,Z\right) $ where the law of $Y$ has an absolutely continuous component and let us observe that  $\mu_{(Y,Z)}\left(dy,dz\right)=\mu_{Z\mid Y=y}(dz)\mu_{Y}(dy)$ where $\mu_{Z\mid Y=y}(dz)$ denotes the conditional law of $Z$ given $Y=y$.	Given $f:\mathbb{R}^2\rightarrow \mathbb{R}_+$  a non-negative measurable function, one has 
	\begin{align*}
	\mathbb{E}&\left(f\left(Y,Z+\xi \right)\right) = \int_{\mathbb{R}^3} f(y,z+\epsilon)\mu_{\left(Y,Z,\xi \right) }\left(dy,dz,d\epsilon \right)= \int_{\mathbb{R}^3} f(y,z+\epsilon)p_\xi\left(\epsilon \right) \mu_{(Y,Z)}\left(dy,dz\right) d\epsilon\\
	&= \int_{\mathbb{R}^3} f(y,z+\epsilon)p_\xi\left(\epsilon \right) \mu_{Z\mid Y=y}(dz)\mu_{Y}(dy) d\epsilon = \int_{\mathbb{R}^3} f(y,x)p_\xi\left(x-z \right) \mu_{Z\mid Y=y}(dz)\mu_{Y}(dy)dx\\
	&=\int_{\mathbb{R}^2}f(y,x) \int_{\mathbb{R}}p_\xi\left(x-z \right) \mu_{Z\mid Y=y}(dz)p_{Y}(y)dydx+\int_{\mathbb{R}^2} f(y,x)\int_{\mathbb{R}}p_\xi\left(x-z \right) \mu_{Z\mid Y=y}(dz)\mu_{Y,s}(dy)dx
	\end{align*}
	where $\int_{\mathbb{R}^2}\int_{\mathbb{R}}p_\xi\left(x-z \right) \mu_{Z\mid Y=y}(dz)p_{Y}(y)dydx$ is positive since $Y$ has an absolutely continuous component and $\xi$ is an absolutely continuous random variable.
	Thus the law of $\left(Y,Z+\xi \right)$ has an absolutely continuous component.
\end{proof}

\begin{proof}[Proof of Lemma \ref{decomposition}]
	Let us first observe that by Fatou's lemma, (\ref{conv.density}) and the fact that $p$ is a probability density, one has
	
	$$\liminf_{M\rightarrow \infty}\int_{\mathbb{R}^d}p_{Y_M}(x) dx \geq \int_{\mathbb{R}^d}\liminf_{M\rightarrow \infty} p_{Y_M}(x)dx \geq \int_{\mathbb{R}^d}p(x)dx = 1.$$
	
	Moreover $\int_{\mathbb{R}^d} p_{Y_M}(x)dx \leq 1 $ for each $M\in \mathbb{N}^*$. Therefore we can conclude that 
	\begin{equation} \label{intpn}
	\lim_{M\rightarrow \infty} \int_{\mathbb{R}^d} p_{Y_M}(x)dx = 1.
	\end{equation}
	As an immediate consequence of (\ref{intpn}) we have that 
	
	\begin{equation}\label{conv.singularpart}
	\lim_{M\rightarrow \infty} \mu_{Y_M,s}(\mathbb{R}^d) = 1-\lim_{M\rightarrow \infty}\int_{\mathbb{R}^d} p_{Y_M}(x)dx = 0. 
	\end{equation} 
	Let us now prove the $L^1$ convergence. For each $M\in \mathbb{N}^*$,
	
	$$\int_{\mathbb{R}^d}\left|p(x)- p_{Y_M}(x)\right| dx =  2\int_{\mathbb{R}^d}\left( p(x)- p_{Y_M}(x)\right)^+dx-\int_{\mathbb{R}^d}\left(  p(x)-p_{Y_M}(x)\right)dx$$
	
	where the second component of the right-hand side goes to $0$ as $M\rightarrow\infty$ by (\ref{intpn}) and the fact that $p$ is a probability density while the first component goes to $0$ as $M\rightarrow\infty$ by Lebesgue
	theorem and the fact that $0 \leq\liminf_{M\rightarrow \infty}\left( p(x)- p_{Y_M}(x)\right)^+\leq\limsup_{M\rightarrow \infty}\left( p(x)- p_{Y_M}(x)\right)^+= \left( \limsup_{M\rightarrow \infty}\left( p(x)- p_{Y_M}(x)\right) \right)^+= 0$.\\ Hence $\lim_{M\rightarrow \infty}\int_{\mathbb{R}^d}\left|p(x)- p_{Y_M}(x)\right| dx =0$.\\
	Let us now prove the second point of the Lemma. For a fixed $M\in\mathbb{N}^*$, one has
	\begin{align*}
	d_{TV}\left(\mu_{Y_M},\mu_{Y} \right)&= \sup_{A\in \mathcal{B}\left(\mathbb{R}^d\right) }\left| \mu_{Y_M}\left(A \right)-\mu_{Y} \left(A \right)\right|= \sup_{A\in \mathcal{B}\left(\mathbb{R}^d \right) }\left| \mu_{Y_M,c}\left(A \right)+\mu_{Y_M,s}\left(A \right)-\mu_{Y} \left(A \right)\right|\\
	&\leq \sup_{A\in \mathcal{B}\left(\mathbb{R}^d \right) }\left|\mu_{Y_M,c}\left(A \right)-\mu_{Y} \left(A \right) \right| + \mu_{Y_M,s}\left(\mathbb{R}^d\right) \leq \int_{\mathbb{R}^d}\left| p_{Y_M}(x)-p(x)\right| dx+ \mu_{Y_M,s}\left(\mathbb{R}^d\right).
	\end{align*}   
	and the right-hand side tends to $0$ as $M\rightarrow\infty$ thanks to what has been proved in the previous steps.	
\end{proof}

\end{document}